\documentclass[11pt,a4paper,reqno]{article}

\usepackage[T1]{fontenc}
\usepackage[margin=2.7cm]{geometry}

\usepackage{amsmath, amsthm, amsfonts, amssymb}
\usepackage{mathrsfs}
\usepackage{ae}	%sostituisce i font raster con font vettoriali
\usepackage[italian,UKenglish]{babel}    % utilizza la lingua italiana come predefinita
\usepackage{amsmath}
\usepackage{amsfonts}
\usepackage{amssymb}
\usepackage{amsthm}
\usepackage{algpseudocode}
\usepackage{float}
\usepackage{xcolor}
\usepackage{mathrsfs}
\usepackage{import}
\usepackage{geometry}
\usepackage{fancyhdr}
\usepackage{lettrine}   % per i capolettere
\usepackage{verbatim}
\usepackage{alltt}

\usepackage{graphics}
\usepackage{graphicx}

\graphicspath{{./imgs/}}

\usepackage{subfigure}
\usepackage{floatflt}

\usepackage[font=small,labelfont=bf]{caption}

\usepackage[utf8x]{inputenc}
\theoremstyle{plain}
\newtheorem{theorem}{Theorem}[section]

\newtheorem{definition}[theorem]{Definition}

\newtheorem{proposition}[theorem]{Proposition}
\newtheorem{lemma}[theorem]{Lemma}
\newtheorem{remark}[theorem]{Remark}

\numberwithin{equation}{section}
\theoremstyle{definition}

\newcommand{\R}{\ensuremath{\mathbb{R}}}

\begin{document}

\title{Global existence and blow-up of solutions to the porous medium equation with reaction and singular coefficients}
%
%\titlerunning{Hamiltonian Mechanics}  % abbreviated title (for running head)
\author{Giulia Meglioli\thanks{Dipartimento di Matematica, Politecnico di Milano, Italia (giulia.meglioli@polimi.it).}}
%\author{Giulia Meglioli*}
%\address{\hbox{\parbox{5.7in}{\medskip\noindent{*Dipartimento di Matematica,\\
%Politecnico di Milano,\\
%   Piazza Leonardo da Vinci 32, 20133 Milano, Italy.
%   \\[3pt]
%        \em{E-mail address: }{\tt
%          giulia.meglioli@polimi.it%\\ \it Corresponding author
%          }}}}}
\date{}
%
%\authorrunning{Giulia Meglioli} % abbreviated author list (for running head)

%
%\institute{Dipartimento di Matematica, Politecnico di Milano Piazza Leonardo da Vinci 32, 20133 Milano, Italy}\\
%\email{giulia.meglioli@polimi.it},\\

\maketitle              % typeset the title of the contribution

\begin{abstract}
We study global in time existence versus blow-up in finite time of solutions to the Cauchy problem for the porous medium equation with a variable density $\rho(x)$ and a power-like reaction term posed in the one dimensional interval $(-R,R)$, $R>0$. Here the weight function is singular at the boundary of the domain $(-R,R)$, indeed it is such that $\rho(x)\sim (R-|x|)^{-q}$ as $|x|\to R$, with $q\ge0$. We show a different behavior of solutions depending on the three cases when $q>2$, $q=2$ and $q<2$. 
\end{abstract}

\bigskip
\bigskip



\section{Introduction}
We are concerned with global existence and blow-up of nonnegative solutions to the Cauchy parabolic problem
\begin{equation}\label{problema}
\begin{cases}
\rho (x) u_t = (u^m)_{xx}+\rho(x) u^p & \text{in }\, (-R,R) \times (0,\tau)=:D_{\tau} \\
u=u_0& \text{in }\, (-R,R) \times \{0\}\,,
\end{cases}
\end{equation}
where $m>1$, $p>1$, $\tau>0$, $R>0$. Let us also set $D_{\infty}=:D$.

We always assume that
\begin{equation}\tag{{\it $H1$}}\label{h1}
\begin{aligned}
&\textrm{(i)} \; u_0\in L^{\infty}((-R,R)), \,\,u_0\geq 0\,\, \textrm{in}\,\, (-R,R)\,;\\
&\textrm{(ii)} \; \rho\in C((-R,R)),\, \rho>0\,\, \textrm{in}\,\, (-R,R)\,.
\end{aligned}
\end{equation}

\medskip

The problem of global solvability of nonlinear evolution problems, such as problem \eqref{problema}, or more generally its counterpart in dimension $N\ge2$, has attracted a lot of interest. We say that a problem is globally solvable in time if it admits a bounded solution for any $t\in (0,+\infty)$. On the contrary, we say that the solution to a given problem blows up in finite time when there exists a time $S>0$ such that
$$
\|u(t)\|_{\infty}\longrightarrow +\infty \quad \text{as}\,\,\,t\to S^{-}\,.
$$
\medskip

Problem \eqref{problema}, with $R=1$ and homogeneous Dirichlet boundary conditions, has been introduced in \cite{KR3} as a mathematical model of evolution of plasma temperature. Here, $u$ is the temperature, $\rho(x)$ is the particle density and $\rho(x) u^p$ represents the volumetric heating of plasma. In \cite{KR1,KR2,KR3}, the authors explain that the interest in thermal waves arises in plasma physics in various laboratory and terrestrial situations where the ambience is at rest but cannot be considered homogeneous. Although, they only study the simpler case of problem \eqref{problema} without the reaction term. They assume that $\rho$ is a positive and smooth function. Let 
$$
M=\int_{-\infty}^{+\infty}\rho(x)\,dx\,,
$$ 
then the authors investigate the behavior of solutions for both the cases of $M<\infty$ and $M=\infty$ showing remarkable differences between them. Only some remarks for the case of problem \eqref{problema} with non-zero reaction term are made in \cite[Section 4]{KR3}. Moreover, in \cite[Introduction]{KR3} a more general source term of the type $A(x) u^p$ has also been considered; however, then the authors do not treat this case.

Strictly related to the latter, is the inhomogeneous porous medium equation in higher dimension space
\begin{equation}\label{eq12}
\rho(x)u_t=\Delta u^m,\quad \text{in}\,\,\,\Omega\times(0,+\infty)\quad (m>1)\,,
\end{equation}
where $\Omega$ is a domain of $\R^N$. Equation \eqref{eq12} with the choices $\Omega\equiv\R^N$ and $\rho$ that decays at infinity as a negative power of $|x|$, is the most studied case, see e.g. \cite{Eid90, EK, I, I2, IM, KKT, KPT, KP1, KP2, KRV10, RV06, RV08, RV09,Vaz07}. Moreover, in \cite{P1} for $N\ge 3$, assuming that $\rho(x)\sim |x|^{-q}$, it is proved that equation \eqref{eq12}, for any bounded initial datum $u_0$, has infinitely many \textit{very weak} solutions if $q>2$ whereas it admits a unique \textit{very weak} solution if $q\le 2$. This different behavior of solutions determines $q=2$ as a threshold value. Concerning well-posedness of degenerate parabolic problems with possibly singular coefficients, we also refer the reader to \cite{PoT,PoPT,PoPT2} and references therein.

Equation \eqref{eq12} can be further generalized to the following weighted PME
\begin{equation}\label{eq13}
\rho_{\nu}(x)u_t=\operatorname{div} \left[\rho_{\mu}(x)\nabla(u^m)\right],\quad \text{in}\,\,\,\Omega\times(0,+\infty)\quad (m>1)\,,
\end{equation}
where $\rho_{\nu}$ and $\rho_{\mu}>0$ are two weights independent of the time variable. With no claim of generality, we refer the reader e.g. to \cite{GMPor}. Depending on the behavior of $\rho_{\nu}$ and $\rho_{\mu}$ as $|x|\to\infty$, existence and uniqueness and the asymptotic behavior of \textit{energy} solutions for large times have been addressed. 
\medskip

We also recall the well known semilinear heat equation defined as follows
\begin{equation}\label{eq14}
u_t=\Delta u+u^p\quad \text{in}\,\,\,\Omega\times(0,T),
\end{equation}
where $T>0$, $\Omega$ is a possibile unbounded domain of $\R^N$ and $p>1$. In particular, we mention the pioneering work by Fujita \cite{F} where global existence and blow-up of solutions to the Cauchy problem associated to \eqref{eq14} is investigated when $\Omega\equiv \R^N$. It is shown that
\begin{itemize}
\item finite time blow-up occurs for all nontrivial nonnegative initial data, for any $$1<p\le 1+\frac{2}{N};$$
\item global existence of solutions for sufficiently small initial data, for any $$p>1+\frac{2}{N}.$$
\end{itemize}
The value $p_c=1+\frac 2N$ is usually referred to as \textit{the Fujita exponent}. We remark that the critical case $p = p_c$ was left open by Fujita, it was proved later in \cite{H}.
\medskip

Closely related to problem \eqref{problema} it is the the following Cauchy problem
\begin{equation}\label{eq11}
\begin{cases}
\rho (x) u_t = \Delta u^m+\rho(x) u^p & \text{in }\, \R^N \times (0,\tau) \\
u=u_0& \text{in }\, \R^N \times \{0\}\,,
\end{cases}
\end{equation}
where $m>1$, $p>1$, $\tau>0$, $N\ge 3$. 
The issue of existence of global in time solutions versus blow-up in finite time for problem \eqref{eq11} has been addressed in \cite{MP1,MP2,MP3}. The authors consider a positive and continuous weight function $\rho(x)$ such that, for some $\varepsilon>0$, $0<k_1\le k_2<+\infty$ and $q\ge0$
\begin{equation}\label{eq11b}
k_1\,|x|^{-q}\,\le\,\rho(x)\,\le\,k_2|x|^{-q}\,\,\,\text{for}\,\,|x|> 1.
\end{equation}

In \cite{MP1}, for $q\in [0, 2)$, the following results have been established. 
\begin{itemize}
\item (\cite[Theorem 2.1]{MP1}) If $p>\overline p$, for a certain $\overline p=\overline p(k_1,k_2,q,m,N)>m$ and if the initial datum is sufficiently small, then solutions exist globally in time. Observe that 
$$\overline p=m+\frac{2-q}{N-q}\,\,\,\text{when}\,\,\,k_1=k_2.$$
\item (\cite[Theorem 2.4]{MP1}) For any $p>1$, for all sufficiently large initial data, solutions blow-up in finite time.
\item (\cite[Theorem 2.6]{MP1}) For $1<p<m$, for any non trivial initial data, solutions blow-up in finite time.
\item (\cite[Theorem 2.7]{MP1}) If  $m<p<\underline p$, for a certain $\underline p=\underline p(k_1,k_2,q,m,N)\le \overline p$, then, for any non trivial initial data, solutions blow-up in finite time, under specific extra assumptions on $\rho$.
$$\underline p\equiv \overline p\,\,\,\text{when}\,\,\,k_1=k_2.$$
\end{itemize}
Such results extend those stated in \cite{SGKM} for problem \eqref{eq11} with $\rho\equiv 1$, $m>1$, $p>1$ (see also \cite{GV}). 

\smallskip
Now, assume that \eqref{eq11b} holds with $q\geq 2$. In \cite{MP2} the following results have been showed.
\begin{itemize}
\item(\cite[Theorem 2.1]{MP2}) If $q=2$ and $p>m$, then, for sufficiently small initial data, solutions exist globally in time.
\item(\cite[Theorem 2.2]{MP2}) If $q=2$ and $p>m$, then, for sufficiently large initial data,  solutions blow-up in finite time.
\item(\cite[Theorem 2.3]{MP2}) If $q>2$, then, for any $p>1$, for sufficiently small initial data, solutions exist globally in time. 
\end{itemize}
\smallskip

Problem \eqref{eq11} has also been considered in \cite{GMP,GMP2}. The authors assume $\rho$ to be such that the Sobolev and the Poincar\'e inequalities hold. They consider initial data belonging to suitable $L^q$ spaces for a finite $q>1$ and they show global existence of solutions corresponding to such data with a smoothing effect on these solutions for the two cases $1<p<m$ and $1<m<p$.  
\smallskip

Furthermore, problem \eqref{eq11} has also been investigated in \cite{MT,MTS}. More precisely, they consider a class of double-nonlinear operators among which equation \eqref{eq11} is included. They show that, (see \cite[Theorem 1]{MT}) if $\rho(x)=|x|^{-q}$ with $q\in (0,2)$, for any $x\in \mathbb R^N\setminus \{0\}$, 
\[p>m+\frac{2-q}{N-q},\]
the initial datum $u_0$ is nonnegative and
$$
\int_{\mathbb R^N}\left\{u_0(x) + [u_0(x)]^{\bar q} \right\}\rho(x) dx<\delta,
$$
for some $\delta>0$ small enough and $\bar q>\frac N2 (p-m)$, then there exists a global solution of problem \eqref{eq11}. In addition, a smoothing estimate holds. On the other hand, if $\rho(x)=|x|^{-q}$
 or $\rho(x)=(1+|x|)^{-q}$ with $q\in [0,2)$, for any initial datum $u_0\not\equiv 0$ and 
 \[p<m+\frac{2-q}{N-q},\] then blow-up prevails, in the sense that there exist $\theta\in (0,1), R>0, T>0$ such that 
 \[\int_{B_R}[u(x,t)]^{\theta} \rho(x) dx\to +\infty \quad \textrm{as}\,\,\, t\to T^-.\]
Such results have also been generalized to more general initial data, decaying at infinity with a certain rate (see \cite{MTS}).

\subsection{Outline of our results}

In what follows, we assume that the function $\rho=\rho(x)$, which is usually referred to as a {\it variable density}, satisfies the following assumption: 
\begin{equation}\tag{{\it$H2$}}\label{h2}
\begin{aligned}
&\,\,\,\,\,\quad\text{for some}\,\,\, \varepsilon_0>0,\,\,\,0<c_1\le c_2<+\infty,\,\,\,q\ge0,\\
&c_1\,(R-|x|)^{-q}\,\le\,\rho(x)\,\le\,c_2(R-|x|)^{-q}\,\,\,\text{for}\,\,\,R-\varepsilon_0<|x|< R.
\end{aligned}
\end{equation}
We shall investigate separately the following three different cases depending on the exponent $q\ge0$ of \eqref{h2}. We say that the decaying rate of $1/\rho(x)$ as $|x|\to R$ is
\begin{itemize}
\item \text{Fast} if $q>2$, \quad\quad\quad\,(Case $1$);\\
\item \text{Critical} if $q=2$, \quad\,\,\,\,(Case $2$);\\
\item \text{Slow} if $0\le q<2$, \quad(Case $3$).
\end{itemize}
\bigskip

We can summarize our results as follows. Assume \eqref{h2} with $q>2$.
\begin{itemize}
\item For $p>m>1$ and for suitable, sufficiently small, initial data $u_0\in L^\infty((-R,R))$, we show the existence of global solutions $u$ such that $u\in L^\infty((-R,R) \times (0, \tau))$ for each $\tau>0$, (see Theorem \ref{teo1}).
\item For any $p>1$ and $m>1$, if $u_0$ is large enough, then the solution to problem \eqref{problema} blows up in a finite time, (see Theorem \ref{teo2}).
\item If $1<p<m$, then for any $u_0\not\equiv 0$, solutions to problem \eqref{problema} blow-up in finite time, (see Theorem \ref{teo6}).
\end{itemize}
Observe that in the latter case, a phenomenon similar to the Fujita phenomenon arises. Here the role of the critical exponent $p_c$ is covered by $p_c=m$.
\medskip

On the other hand, if we assume \eqref{h2} with $q=2$, then we show that
\begin{itemize}
\item For any $p>m$, if $u_0$ is sufficiently large, then solutions to problem \eqref{problema} blow-up in finite time, (see Theorem \ref{teo3}).
\item For any $p>m$, if $u_0$ has compact support and it is small enough, then, under suitable assumptions on $c_1$ and $c_2$, there exist global in time solutions to problem \eqref{problema}, which belong to $L^\infty((-R,R) \times (0, +\infty))$, (see Theorem \ref{teo4}).
\end{itemize}
\medskip

Now, assume \eqref{h2} with $0\le q<2$. In Theorem \ref{teo5} we show that 
\begin{itemize}
\item For $1<p<m$ and for suitable initial data $u_0\in L^\infty((-R,R))$, we show the existence of global solutions belonging to $L^\infty((-R,R) \times (0, \tau))$ for each $\tau>0.$
\item For $p>m>1$, if $u_0$ satisfies a suitable decaying condition as $|x|\to R$, then problem \eqref{problema} admits a solution in $L^\infty((-R,R)\times (0, +\infty))$. 
\end{itemize}

\smallskip

The proofs mainly rely on suitable comparison principles and properly constructed sub- and supersolutions, 
which crucially depend on the behavior near the boundary of the interval $(-R,R)$ of the density function $\rho(x)$. More precisely,  they are of the type
\begin{equation}\label{e4f}
w(x,t)=C\zeta(t)\left [ 1- \frac{\mathfrak{s}(x)}{a} \eta(t) \right ]_{+}^{\frac{1}{m-1}}\quad \textrm{for any}\,\,\, (x,t)\in (-R,R)\times [0, T),
\end{equation}
for suitable functions $\mathfrak{s}(x)$, $\zeta=\zeta(t), \eta=\eta(t)$ and constants $C>0$, $a>0$. 
The paper is organized as follows. In Section \ref{statements} we state our main results, in Section \ref{prel} we give the precise definitions of solutions, we discuss uniqueness of solution depending on the rate of explosion of $\rho(x)$ as $|x|\to R$ and we recall some auxiliary results. In Section \ref{Global} we prove Theorem \ref{teo1} and Theorem \ref{teo2}, which concern the case of fast decaying rate, i.e. $q>2$. The results obtained for the critical case $q=2$, (that is, Theorems \ref{teo3}, \ref{teo4}) are proved in Section \ref{critical}. Finally, in Section \ref{subcritical} Theorem \ref{teo5} of the slow decaying rate, $0\le q<2$, is proved .

\section{Statements of the main results}\label{statements}

Observe that, in view of \eqref{h2}-(ii), there exist $\rho_1, \rho_2\in (0,+\infty)$ with $\rho_1\le\rho_2$  such that
\begin{equation}\label{eq21}
\rho_1\le\rho(x)\le\rho_2\quad \text{for all}\,\,\,x\in [-R+\varepsilon_0, R-\varepsilon_0]\,.
\end{equation}

\subsection{Case $\mathbf{1}$: fast decaying rate}

In what follow we assume that \eqref{h2} holds with $$q>2.$$ We introduce the parameter $b\in \R$ which is defined as follows
\begin{equation}\label{eq25}
b:=q-2,
\end{equation}

The first result concerns the global existence of solutions to problem \eqref{problema}, for any $m>1$ and $p> m$. For some $\varepsilon>0$, let \begin{equation}\label{eq210}
\mathfrak{s}(x):=\begin{cases}
(R-|x|)^{-b} \quad \quad \quad \quad &\text{if}\,\,\,R-\varepsilon<|x|\le R \\ 
\dfrac{2\varepsilon-b(R-\varepsilon)+\frac{b|x|^2}{R-\varepsilon}}{2\varepsilon^{b+1}} \quad &\text{if}\,\,\,x\in [-R+\varepsilon, R-\varepsilon].
\end{cases}
\end{equation}

\begin{theorem}\label{teo1}
Let assumptions \eqref{h1} and \eqref{h2} with $q>2$ %and \eqref{eq24} 
be satisfied. Moreover assume that

\begin{equation}\label{kbound}
\frac{c_2}{c_1}<m+\frac{m-1}{b}. %\min\left\{ \frac{m-1}{b}\left(\frac{T^{\beta}\varepsilon^{b+1}a}{R-\varepsilon}\right)\right\}.
\end{equation}
Suppose that $$p>m>1\,,$$
and that $u_0$ is small enough. Then problem \eqref{problema} admits a global solution $u\in L^\infty(D_{\tau})$ for any $\tau>0$.
More precisely, if $C>0$ is sufficiently small, $T>0$, $a> 0$ with
$$
\omega_0\le\frac{C^{m-1}}{a}\le\omega_1,
$$
for suitable $0<\omega_0<\omega_1$,
\begin{equation}\label{eq26}
\alpha=\frac{1}{p-1},\quad \beta=\frac{p-m}{p-1},
\end{equation}
\begin{equation}\label{eq27}
u_0(x) \le CT^{-\alpha} \left [ 1-\frac{\mathfrak{s}(x)}{a} \,T^{-\beta}\right ]^{-\frac{1}{m-1}} \quad \text{for any}\,\, x\in (-R,R)\,,
\end{equation}
then problem \eqref{problema} admits a global solution $u$, which satisfies the bound from above
\begin{equation}\label{eq28}
u(x,t) \le C(T+t)^{-\alpha} \left [ 1-\frac{\mathfrak{s}(x)}{a}(T+t)^{-\beta} \right ]^{-\frac{1}{m-1}} \, \text{for any}\,\, (x,t)\in D\,.
\end{equation}
\end{theorem}
The precise choice of parameters $C>0$, $T>0$, $a>0$ in Theorem \ref{teo1} is discussed in Remark \ref{rem1} below.

The next result concerns the blow-up of solutions in finite time, for every $p>1$ and $m>1$, provided that the initial datum is sufficiently large.

\begin{theorem}\label{teo2}
Let assumptions \eqref{h1} and \eqref{h2} with $q>2$. 
For any $$p>1,\quad m>1,$$ and for any $T>0$, if the initial datum $u_0$ is large enough, then the solution $u$ of problem \eqref{problema} blows up in a finite time $S\in (0,T]$, in the sense that
\begin{equation}\label{eq29}
\|u(t)\|_{\infty} \to \infty \text{ as } t \to S^{-}\,.
\end{equation}
More precisely, we have the following three cases
\begin{itemize}
\item[(a)] Let $p>m$. If $C>0$, $a>0$ are large enough, $T>0$,
\begin{equation}\label{eq211}
u_0(x)\ge CT^{-\frac{1}{p-1}}\left[1-\frac{\mathfrak{s}(x)}{a}\,T^{\frac{m-p}{p-1}}\right]^{\frac{1}{m-1}}_{+} \quad \text{for any}\,\, x\in (-R,R)\,,
\end{equation}
then the solution $u$ of problem \eqref{problema} blows up and satisfies the bound from below
\begin{equation} \label{eq212}
u(x,t) \ge C (T-t)^{-\frac{1}{p-1}}\left [1- \frac{\mathfrak{s}(x)}{a}\, (T-t)^{\frac{m-p}{p-1}} \right ]_{+}^{\frac{1}{m-1}}\,\, \text{for any}\,\, (x,t) \in D_S\,.
\end{equation}
\item[(b)] Let $p<m$. If $\frac{C^{m-1}}{a}>0$ and $a>0$ are big enough, $T>0$ and \eqref{eq211} holds, then the solution $u$ of problem \eqref{problema} blows up and satisfies the bound from below \eqref{eq212}.
\item[(c)] Let $p=m$. If $\frac{C^{m-1}}{a}>0$ and $a>0$ are big enough, $T>0$ and \eqref{eq211} holds, then the solution $u$ of problem \eqref{problema} blows up and satisfies the bound from below \eqref{eq212}.
\end{itemize}
\end{theorem}

Observe that if $u_0$ satisfies \eqref{eq211}, then
\begin{equation*}
\operatorname{supp}u_0\supseteq \left\{x\in (-R,R)\,:\, \mathfrak{s}(x)< a T^{\frac{p-m}{p-1}}\right\}\,.
\end{equation*}
From \eqref{eq212} we can infer that
\begin{equation}\label{eq213}
\operatorname{supp}u(\cdot, t)\supseteq \left\{x\in (-R,R)\,:\, \mathfrak{s}(x) < a (T-t)^{\frac{p-m}{p-1}}\right\}\quad \textrm{for all } t\in[0,S)\,.
\end{equation}
\medskip
We remark that the supports of $u_0$ and $u(\cdot,t)$ for any $t>0$, are nonempty if one choses $$0<\varepsilon_0=\varepsilon<\frac{b}{b+2}R,$$ where $\varepsilon_0$ and $\varepsilon$ are introduced in \eqref{h2} and \eqref{eq210} respectively. This choice is always admissible. Then the constants $c_1$ and $c_2$ in assumption \eqref{h2} will be fixed accordingly. The choice of the parameters $C>0$, $T>0$ and $a>0$ is discussed in Remark \ref{rem2}.

\begin{theorem}\label{teo6}
Let assumptions \eqref{h1} and \eqref{h2} with $q>2$ be satisfied.  Suppose that $$1<p< m\,,$$ and that $u_0\in C((-R,R)), u_0(x) \not\equiv 0$. Then, for any sufficiently large $T>0$, the solution $u$ of problem \eqref{problema} blows up in a finite time $S\in (0,  T]$, in the sense that $$\|u(t)\|_{\infty} \to+ \infty \quad \textrm{as} \,\,\, t \to S^{-}.$$ More precisely, the bound from below \eqref{eq212} holds, with $b, C, a, \varepsilon$ as in Theorem \ref{teo2}-(b)\,.
\end{theorem}

\subsection{Case $\mathbf{2}$: critical decaying rate}

In what follows we assume that \eqref{h2} holds with $$q=2.$$ Moreover, observe that, due to \eqref{h1}, assumption \eqref{eq21} holds. 
\medskip

The first result concerns the global existence of solutions to problem \eqref{problema} for $p>m$. Let $\delta$ be a positive constant then we state the following theorem.

\begin{theorem}\label{teo3}
Assume that \eqref{h1} and \eqref{h2} for $q=2$ hold. Furthermore, suppose that $$p>m\,,$$
and that $u_0$ is small enough and has compact support. Then problem \eqref{problema} admits a global solution $u\in L^\infty((-R,R)\times (0, +\infty))$. \newline
More precisely, for any $\delta>0$, if $C>0$ and $a>0$ are small enough and
\begin{equation}\label{eq20}
u_0(x) \le CT^{-\frac{1}{p-1}} \left [ 1- \frac{(R-|x|)^{-\delta}}{a} \, T^{-\frac{p-m}{p-1}} \right ]_{+}^{\frac{1}{m-1}} \quad \text{for any}\,\, x\in (-R,R)\,,
\end{equation}
then problem \eqref{problema} admits a global solution $u\in L^{\infty}(\R^N\times (0,+\infty))$. Moreover,
\begin{equation}\label{eq21b}
u(x,t) \le C(T+t)^{-\frac{1}{p-1}} \left [ 1- \frac{(R-|x|)^{-\delta}}{a} \,(T+t)^{-\frac{p-m}{p-1}} \right ]_{+}^{\frac{1}{m-1}} \, \text{for any}\,\, (x,t)\in(-R,R) \times (0,+\infty)\,.
\end{equation}
\end{theorem}

Observe that if $u_0$ satisfies \eqref{eq20}, then
\begin{equation*}
\operatorname{supp}u_0\subseteq \{x\in (-R,R)\,:\, (R-|x|)^{-\delta}\leq a T^{\frac{p-m}{p-1}}\}\,.
\end{equation*}
From \eqref{eq21b} we can infer that
$$
\operatorname{supp}u(\cdot, t)\subseteq \{x\in (-R,R)\,:\, (R-|x|)^{-\delta}\leq a (T+t)^{\frac{p-m}{p-1}}\}\quad \textrm{for all } t>0\,.
$$
\medskip

The choice of the parameters $C>0, T>0$ and $a>0$ is discussed in Remark \ref{rem72}.

\bigskip

The next result concerns the blow-up of solutions in finite time, for every $p>m>1$, provided that the initial datum is sufficiently large.

Let
\[\mathfrak{t}(x):=\begin{cases}
\log(R-|x|)  &\quad \text{if}\quad  x\in (-R,-R+\varepsilon)\cup(R,R-\varepsilon), \\
& \\
\dfrac{(R-\varepsilon)^2-|x|^2+\log(\varepsilon)[2\varepsilon(R-\varepsilon)]}{2\varepsilon(R-\varepsilon)} &\quad\text{if}\quad  x\in [-R+\varepsilon,R-\varepsilon]\,.
\end{cases}\]

\begin{theorem}\label{teo4}

Let assumption \eqref{h1} and \eqref{h2} with $q=2$ hold. For any $$p>m$$ and for any $T>0$, if the initial datum $u_0$ is large enough, then the solution $u$ of problem \eqref{problema} blows up in a finite time $S\in (0,T]$, in the sense that
\begin{equation}\label{eq218}
\|u(t)\|_{\infty} \to \infty \text{ as } t \to S^{-}\,.
\end{equation}
More precisely, if $C>0$ and $a>0$ are large enough, $T>0$,
\begin{equation}\label{eq219}
u_0(x)\ge CT^{-\frac{1}{p-1}}\left[1-\frac{\mathfrak{t}(x)}{a}\,T^{\frac{p-m}{p-1}}\right]^{\frac{1}{m-1}}_{+}\,, \quad \text{for any}\,\, x\in (-R,R)\,,
\end{equation}
then the solution $u$ of problem \eqref{problema} blows up and satisfies the bound from below
\begin{equation} \label{eq220}
u(x,t) \ge C (T-t)^{-\frac{1}{p-1}}\left [1- \frac{\mathfrak{t}(x)}{a}\, (T-t)^{\frac{p-m}{p-1}} \right ]_{+}^{\frac{1}{m-1}}, \quad \text{for any}\,\, (x,t) \in D_S.
\end{equation}
\end{theorem}

\medskip
Observe that if $u_0$ satisfies \eqref{eq219}, then
\begin{equation*}
\operatorname{supp}u_0\supseteq \{x\in (-R,R)\,:\, \mathfrak{t}(x)< a T^{-\frac{p-m}{p-1}}\}\,.
\end{equation*}
From \eqref{eq220} we can infer that
\begin{equation}\label{eq221}
\operatorname{supp}u(\cdot, t)\supseteq \{x\in (-R,R)\,:\, \mathfrak{t}(x) < a (T-t)^{-\frac{p-m}{p-1}}\}\quad \textrm{for all } t\in[0,S)\,.
\end{equation}
\medskip
The choice of the parameters $C>0, T>0$ and $a>0$ is discussed in Remark \ref{rem71}.

\subsection{Case $\mathbf{3}$: slow decaying rate}

In what follows we assume that \eqref{h2} holds with $$0\le q<2.$$ The next result concerns the global existence of solutions to problem \eqref{problema} for any $p>1$ and $m>1$, $p\neq m$. Let us introduce the parameter $d\in \R$ such that
\begin{equation}\label{hd}
0< d<\min \{2-q\,,\,\,1\}\,.
\end{equation}

\begin{theorem}\label{teo5}
Let assumptions \eqref{h1} and \eqref{h2} be satisfied with $0\le q<2$. Suppose that $$1<p<m\,,\quad \text{or}\,\,\, p> m\geq 1\,,$$
and that $u_0$ is small enough. Then problem \eqref{problema} admits a global solution $u\in L^\infty(D_{\tau})$ for any $\tau>0$.
More precisely, we have the following cases.
\begin{itemize}
\item[(a)]\, Let $1<p<m$. If $C>0$ is big enough, $T>1$, $\alpha> 0$,
\begin{equation}\label{eq215}
u_0(x) \le CT^{\alpha} \left (R-|x| \right )^{\frac{d}{m}} \quad \text{for any}\,\, x\in (-R,R)\,,
\end{equation}
then problem \eqref{problema} admits a global solution $u$, which satisfies the bound from above
\begin{equation}\label{eq216}
u(x,t) \le C(T+t)^{\alpha} \left (R-|x| \right )^{\frac{d}{m}} \, \text{for any}\,\, (x,t)\in D\,.
\end{equation}
\item[(b)]\, Let $p> m \geq  1$. If $C>0$ is small enough, $T>0$ and \eqref{eq215} holds with $\alpha=0$, then problem \eqref{problema} admits a global solution $u\in L^{\infty}(D)$, which satisfies the bound from above \eqref{eq216} with $\alpha=0$.
\end{itemize}
\end{theorem}

Note that in Theorem \ref{teo5} we do not require that $\operatorname{supp}\, u_0$ is compact.
\medskip

The choice of the parameters $C>0$, $T>0$ and $a>0$ is discussed in Remark \ref{rem71}.

\section{Preliminaries and well-posedness results}\label{prel}
In this section we give the precise definition of solutions of problem \eqref{problema} and we address well-posedness of problem \eqref{problema} depending on the exponent $q\ge0$ of assumption \eqref{h2}. Furthermore, we recall some auxiliary results that will be used in the subsequent sections. 

\medskip

Throughout the paper we deal with {\em very weak} solutions to problem \eqref{problema} according to the following definition. 

\begin{definition}\label{def1}
Let $I:=(-R,R)$, $u_0\in L^{\infty}(I)$ with $u_0\ge0$. Let $\tau>0$, $p>1$, $m>1$. By a solution to problem \eqref{problema} in $I\times(0,\tau)$ we mean any nonnegative function $u\in L^{\infty}(I\times(0,S))$, for any $ S<\tau$ such that
\begin{equation*}\label{eq31}
\begin{aligned}
-\int_{0}^{\tau}\int_{I} \rho(x) u \varphi_t dx\,dt &= \int_{I}\rho(x) u_0(x) \varphi(x,0) \,dx \\ &+ \int_{0}^{\tau}\int_{I}  u^m \varphi_{xx} \,dx\,dt \\ &+ \int_{0}^{\tau}\int_{I} \rho(x) u^p \varphi \,dx\,dt
\end{aligned}
\end{equation*}
for any $\varphi \in C_c^{\infty}( \overline{I}\times[0,\tau))$, $\varphi\ge0$, $\varphi(-R,t)=\varphi(R,t)=0$ for any $t\in[0,\tau)$. Subsolutions (respectively supersolutions) of \eqref{problema} are defined replacing $"="$ by $"\le"$ (respectively $"\ge"$) in equality \eqref{eq31}.
\end{definition}

For every $0<\delta<R$, define $I_{\delta}=(-R+\delta,R-\delta)$. Observe that $$I_{\delta}\longrightarrow I\quad \text{as}\,\,\,\delta\to 0.$$ Then we consider the auxiliary problem
\begin{equation}\label{eq32}
\begin{cases}
 u_t=\frac 1{\rho(x)}(u^m)_{xx} +u^p & \text{ in } I_{\delta}\times(0,\tau) \\
u=0  &  \text{ on } \{-R_{\delta}, R_{\delta}\} \times(0,\tau)\\
u=u_0 &  \text{ in } I_{\delta}\times\{0\}\,.\\
\end{cases}
\end{equation}

\begin{proposition}\label{prop31}
Let hypothesis \eqref{h1} be satisfied. Then there exists a solution $u$ to problem \eqref{eq32} with
$$\tau\geq \tau_{\delta}:=\frac{1}{(p-1)\|u_0\|_{L^\infty(I_{\delta})}^{p-1}}.$$
\end{proposition}
The proof of Proposition \ref{prop31} can be found in \cite[Proposition 3.3]{MP1}. Moreover, the following comparison principle for problem \eqref{eq32} holds (see \cite{ACP} for the proof).
\begin{proposition}\label{prop32}
Let assumption \eqref{h1} holds.
If $u$ is a subsolution of problem \eqref{eq32} and $v$ is a supersolution of \eqref{eq32}, then
$$u\le v \quad  \textrm{a.e. in } \, I_{\delta} \times (0,\tau).$$
\end{proposition}

\begin{proposition}
Let hypotheses \eqref{h1} be satisfied. Then there exists a solution $u$ to problem \eqref{problema} with
$$\tau\geq \tau_0:=\frac{1}{(p-1)\|u_0\|_{L^{\infty}((-R,R))}^{p-1}}.$$
Moreover, $u$ is the {\em minimal solution}, in the sense that for any solution $v$ to problem \eqref{problema} there holds
\[u\leq v \quad \textrm{in }\,\,\, \mathbb (-R,R)\times (0, \tau)\,.\]
\label{prop1}
\end{proposition}

The proof of the latter statement is the same of \cite[Proposition 3.5]{MP1}. Concerning uniqueness of solution to problem \eqref{problema}, we state the following Propositions.

\begin{proposition}\label{propuniq}
Let assumption \eqref{h2} be satisfied with $$q\ge 2.$$ Then there exists at most one bounded solution $u$ to problem \eqref{problema}.
\end{proposition}
\begin{proposition}\label{nonuniq}
Let hypothesis \eqref{h2} be satisfied with $$q<2.$$ If there exists a supersolution $V>0$ of problem 
\begin{equation}
\begin{aligned}
\dfrac{1}{\rho} V_{xx}&=-1 \\
\lim_{|x|\to R}V(x)&=0\,,
\end{aligned}
\end{equation}
then there exist infinitely many solutions $u$ of problem \eqref{problema} that belong to $L^{\infty}((-R,R)\times(0,T])$, for some $T>0$. In particular, for any $c>0$, there exists a solution $u_c$ of problem \eqref{problema} such that
$$
\lim_{|x|\to R}\frac{1}{T}\int_{0}^{T}u_c^{m}(x,t)\,dt=c
$$
\end{proposition}

The proofs of Propositions \ref{propuniq}, \ref{nonuniq} can be proved by standard methods (for a complete proof see for instance \cite[Propositions 1.7.1, 2.5.1]{Tesi}).

In conclusion, we can state the following two comparison results for the solution to problem \eqref{problema}, which will be used in the sequel.

\begin{proposition}\label{cpsup}
Let hypothesis \eqref{h1} be satisfied. Let $\bar{u}$ be a supersolution to problem \eqref{problema}. Then, if $u$ is the minimal solution to problem \eqref{problema} given by Proposition \ref{prop1}, then
\begin{equation}\label{eq34}
u\le\bar{u} \quad \text{a.e. in } (-R,R) \times (0,\tau)\,.
\end{equation}
In particular, if $\bar{u}$ exists until time $\tau$, then also $u$ exists at least until time $\tau$.
\end{proposition}
\begin{proof}
For any $\delta>0$, $\overline{u}$ is a supersolution to problem \eqref{eq31}. Hence, by Proposition \ref{prop31},
$$
u_{\delta}\le\overline{u}\quad \text{in}\,\,\,I_{\delta}\times(0,\tau).
$$
By passing to the limit as $\delta\to0$, we obtain \eqref{eq34} which ensures that $u$ does exist at least up to time $\tau$, by definition of maximal existence time.
\end{proof}

\begin{proposition}\label{cpsub}
Let hypothesis \eqref{h1} be satisfied. Let $u$ be a solution to problem \eqref{problema} for some time $\tau=\tau_1>0$ and $\underline{u}$ a subsolution to problem \eqref{problema} for some time $\tau=\tau_2>0$. Suppose also that
$$
\operatorname{supp }\underline{u}|_{\overline{D}_S} \text{ is compact for every }  \, S\in (0, \tau_2)\,.
$$
Then
\begin{equation}\label{eq35}
u\ge\underline{u} \quad \text{ in }\,\, (-R,R) \times \left(0,\min\{\tau_1,\tau_2\}\right)\,.
\end{equation}
\end{proposition}
\begin{proof}
For a fixed $S <\min\{\tau_1,\tau_2\}$, we can take $\delta>0$ sufficiently small that
$$\text{supp}\underline{u}|_{D_S}\subset I_{\delta}\times[0,S].$$
Then $u$ and $\underline{u}$ are a super- and a subsolution, respectively, to problem \eqref{eq31}. Hence
$$ u\ge\underline{u}\quad\text{in}\,\,\,I_{\delta}\times(0,S).$$
By passing to the limit as $\delta\to0$, due to the arbitrariness of $S$, we obtain \eqref{eq35}.
\end{proof}

In what follows we also consider solutions of equations of the form
\begin{equation}\label{eq36}
u_t = \frac 1{\rho(x)}(u^m)_{xx} + u^p \quad \textrm{in }\,\, I\times (0, \tau),
\end{equation}
where $I\subseteq\mathbb (-R,R)$ is an open interval. Solutions are meant in the sense of Definition \ref{def1} where the nonnegative test functions $\varphi\in C_c^{\infty}(\overline{I}\times[0,\tau))$ are such that $\varphi|_{\partial I}=0$ for any $t\in[0,\tau)$.

We now recall the following well-known criterion, that will be used in the sequel.
Let $I\subseteq \mathbb R$ be an open interval. Suppose that $I=I_1\cup I_2$ with  $I_1\cap I_2=\emptyset$, and that   $\Sigma:=\partial I_1\cap\partial I_2$ is of class $C^1$. Let $n$ be the unit outwards normal to $I_1$ at $\Sigma$.
Let
\begin{equation}\label{eq38}
u=\begin{cases}
u_1 & \textrm{in }\, I_1\times [0, \tau),\\
u_2 & \textrm{in }\, I_2\times [0, \tau)\,,
\end{cases}
\end{equation}
where $\partial_t u\in C(I_1\times (0, \tau)), u_1^m\in C^2(I_1\times (0, \tau))\cap C^1(\overline{I}_1\times (0, \tau)) , \partial_t u_2\in C(I_2\times (0, \tau)),u_2^m\in C^2(I_2\times (0, \tau))\cap C^1(\overline{I}_2\times (0, \tau)).$

\begin{lemma}\label{lemext}
Let assumption \eqref{h1} be satisfied.

(i) Suppose that
\begin{equation}\label{eq39}
\begin{aligned}
&\partial_t u_1 \geq \frac 1{\rho} (u_1^m)_{xx} +u_1^p \quad \textrm{for any}\,\,\, (x,t)\in I_1\times (0, \tau),\\
&\partial_t u_2 \geq  \frac 1{\rho} (u_2^m)_{xx} + u_2^p \quad \textrm{for any}\,\,\, (x,t)\in I_2\times (0, \tau),
\end{aligned}
\end{equation}
\begin{equation}\label{eq310}
u_1=u_2, \quad \frac{\partial u_1^m}{\partial n}\geq \frac{\partial u_2^m}{\partial n}\quad \textrm{for any }\,\, (x,t)\in \Sigma\times (0, \tau)\,.
\end{equation}
Then $u$, defined in \eqref{eq38}, is a supersolution to equation \eqref{eq36}, in the sense of Definition \ref{def1} with $I_R$ replaced by $I$.

(ii)  Suppose that
\begin{equation*}\label{eq185b}
\begin{aligned}
&\partial_t u_1 \leq  \frac 1{\rho} (u_1^m)_{xx} +u_1^p \quad \textrm{for any}\,\,\, (x,t)\in I_1\times (0, \tau),\\
&\partial_t u_2 \leq  \frac 1{\rho} (u_2^m)_{xx} + u_2^p \quad \textrm{for any}\,\,\, (x,t)\in I_2\times (0, \tau),
\end{aligned}
\end{equation*}
\begin{equation*}\label{eq311}
u_1=u_2, \quad (u_1^m)_x \leq (u_2^m)_x \quad \textrm{for any}\,\, (x,t)\in \Sigma\times (0, \tau)\,.
\end{equation*}
Then $u$, defined in \eqref{eq38}, is a subsolution to equation \eqref{eq36}, in the sense of Definition \ref{def1} with $I_R$ replaced by $I$.
\end{lemma}
The proof of Lemma \ref{lemext} is the same as the proof of \cite[Lemma 3.1]{MP1} for the special case $N=1$.

\section{Proofs of Theorems \ref{teo1}, \ref{teo2} and \ref{teo6}} \label{Global}

Le us consider the function

\begin{equation}\label{eq42b}
w(x,t):=
\begin{cases} 
u(x,t)\quad& (x,t)\in (-R,-R+\varepsilon)\times(0,T), \\ 
v(x,t) \quad & (x,t)\in [-R+\varepsilon,R-\varepsilon]\times(0,T),\\ 
u(x,t)\quad& (x,t)\in (R-\varepsilon,R)\times(0,T) ;
\end{cases}
\end{equation}
\begin{equation}\label{eq42}
u(x,t):=C\zeta(t)\left[F(x,t)\right]_{+}^{-\frac{1}{m-1}}; %\quad \text{for}\,\, %(x,t)\in [(-R,-R+\varepsilon)\cup(R-\varepsilon,R)]\times(0,+\infty),
\end{equation}
\begin{equation}\label{eq42a}
v(x,t):=C\zeta(t)\left[G(x,t)\right]_{+}^{-\frac{1}{m-1}}; %\quad \text{for}\,\, (x,t)\in (-R+\varepsilon,R-\varepsilon)\times(0,+\infty),
\end{equation}
where 
\begin{equation}\label{eq42c}
\begin{aligned}
&F(x,t):= 1-\,\frac{(R-|x|)^{-b}}{a}\eta(t)\,;\\ %\quad\quad \quad \text{for any}\,\, R-\varepsilon<|x|\le R;\\
&G(x,t):=1-\left[\frac{2\varepsilon-b(R-\varepsilon)}{2\varepsilon^{b+1}}+\frac{b|x|^2}{2\varepsilon^{b+1}(R-\varepsilon)}\right]\, \frac{\eta(t)}{a} ;%\quad \text{for any}\,\, -R+\varepsilon<x\le R-\varepsilon;
\end{aligned}
\end{equation}
for suitable positive $\zeta$ and $\eta$ and for $C > 0$, $a>0$, $b$ as in \eqref{eq25}. Observe that, for any $(x,t) \in [(-R,-R+\varepsilon)\cup (R-\varepsilon,R)]\times(0,T)$, we have

\begin{equation}\label{eq43}
\begin{aligned}
u_t &=C\,\zeta'\,F^{\frac{1}{m-1}}-\frac{1}{m-1}C\zeta\,F^{\frac{1}{m-1}-1}\frac{(R-|x|)^{-b}}{a}\eta'\\
&=C\,\zeta'\,F^{\frac{1}{m-1}}+\frac{1}{m-1}C\zeta\,F^{\frac{1}{m-1}}\frac{\eta'}{\eta}-\frac{1}{m-1}C\zeta\,F^{\frac{1}{m-1}-1}\frac{\eta'}{\eta}.
\end{aligned}
\end{equation}
\begin{equation}\label{eq44}
\begin{aligned}
(u^m)_{xx}
&=-C^m \,\zeta^m\frac{m}{(m-1)} \,\frac{b}{a}\,(R-|x|)^{-b-2}\,\eta\left(\frac{bm}{m-1}+1\right)\,F^{\frac{1}{m-1}}\\
&+C^m \,\zeta^m\frac{m}{(m-1)^2} \,\frac{b^2}{a}\,(R-|x|)^{-b-2}\,\eta\,F^{\frac{1}{m-1}-1}.\\
\end{aligned}
\end{equation}
On the other hand, for any $(x,t) \in [-R+\varepsilon,R-\varepsilon]\times(0,T)$ 
\begin{equation}\label{eq45}
\begin{aligned}
v_t &=C\,\zeta'\,G^{\frac{1}{m-1}}+\frac{1}{m-1}C\zeta\,G^{\frac{1}{m-1}}\frac{\eta'}{\eta}-\frac{1}{m-1}C\zeta\,G^{\frac{1}{m-1}-1}\frac{\eta'}{\eta}.
\end{aligned}
\end{equation}
\begin{equation}\label{eq46}
\begin{aligned}
(v^m)_{xx}
&=-\frac{C^m}{a} \,\zeta^m\frac{m}{(m-1)} \,\varepsilon^{-b-1}\frac{b}{R-\varepsilon}\,G^{\frac{1}{m-1}}\\
&+C^m \,\zeta^m\frac{m}{(m-1)^2} \,\varepsilon^{-2b-2}\frac{b^2}{(R-\varepsilon)^2}\,\frac{x^2}{a^2}\,\eta^2\,G^{\frac{1}{m-1}-1}.\\
\end{aligned}
\end{equation}

\subsection{Proof of global existence for the fast decaying rate}
We want to construct a suitable family of supersolutions of equation
\begin{equation}\label{eq41}
u_t =\frac{1}{\rho(x)}(u^m)_{xx}+u^p \quad \text{ in } (-R,R)\times(0,+\infty)=:D.
\end{equation}
For this purpose, we set, 
\begin{equation}\label{eq41a}
\bar{u}(x,t)\equiv w(x,t) \quad \text{for all}\,\, (x,t)\in D.
\end{equation}
for $w$ as in \eqref{eq42b} and $\zeta,\eta\in C^1([0,+\infty);[0,+\infty))$. Moreover, we define
\begin{equation}\label{A}
\begin{aligned}
&A:=\{(x,t)\in [(-R,-R+\varepsilon)\cup (R-\varepsilon,R)]\times(0,+\infty):0<F(x,t)<1\};\\
\end{aligned}
\end{equation}
and
\begin{equation}\label{eq47}
\begin{aligned}
& \bar{\sigma}(t) := \zeta ' +  \frac{\zeta}{m-1} \frac{\eta'}{\eta} + \frac{C^{m-1}}{a}\, \zeta^m\,\frac{m}{m-1}\frac{b}{c_2}\left(\frac{b\,m}{m-1}+1\right)\eta;\\
& \bar{\lambda}(t) := \frac{\zeta}{m-1} \frac{\eta'}{\eta} +  \frac{C^{m-1}}{a}\,\zeta^m\,\frac{m}{(m-1)^2}\,\frac{b^2}{c_1}\,\eta; \\
& \bar{\gamma}(t):=C^{p-1} \zeta^p;\\
& \bar{\sigma}_0(t) := \zeta ' +  \frac{\zeta}{m-1} \frac{\eta'}{\eta} + \frac{C^{m-1}}{a}\, \zeta^m\,\frac{m}{m-1}\frac{b}{\rho_2}\frac{\varepsilon^{-b-1}}{R-\varepsilon}\eta;\\
& \bar{\lambda}_0(t) := \frac{\zeta}{m-1} \frac{\eta'}{\eta} +  \frac{C^{m-1}}{a^2}\,\zeta^m\,\frac{m}{(m-1)^2}\,\frac{b^2}{\rho_1}\,\varepsilon^{-2b-2}\,\eta^2. \\
\end{aligned}
\end{equation}
for $c_1,c_2$ as in \eqref{h2}.

\begin{proposition}\label{prop41}
Let $\zeta,\eta\in C^1([0,+\infty); [0, +\infty))$, $\zeta'\leq 0$. Assume \eqref{h2}, \eqref{eq21}, %\eqref{eq22}, 
\eqref{kbound} and that
\begin{equation}\label{eq48a}
\eta<\,a\,\varepsilon^b\min\left\{1,\frac{(m-1)}{b}\frac{\rho_1}{\rho_2}\frac{\varepsilon}{R-\varepsilon}\right\};
\end{equation}
\begin{equation}\label{eq48}
-\frac{\eta'}{\eta^2}\ge \,\frac{C^{m-1}}{a}\zeta^{m-1}\frac{m}{m-1}\,\frac{b^2}{c_1};
\end{equation}
\begin{equation}\label{eq49}
\zeta'+\frac{C^{m-1}}{a}\zeta^m\frac{m}{m-1}b\,\left[\frac{1}{c_2}\left(\frac{bm}{m-1}+1\right)-\frac{b}{c_1(m-1)}\right] \eta-C^{p-1}\zeta^p\,\ge 0\,;
\end{equation}
\begin{equation}\label{eq49a}
-\frac{\eta'}{\eta^3}\ge \,\frac{C^{m-1}}{a}\zeta^{m-1}\frac{m}{m-1}\,b^2\frac{\varepsilon^{-2b-2}}{a\,\rho_1};
\end{equation}
\begin{equation}\label{eq49b}
\zeta'+\frac{C^{m-1}}{a}\zeta^m\frac{m}{m-1}b\,\varepsilon^{-b-1}\,\left[\frac{1}{\rho_2(R-\varepsilon)}-\frac{b\,\varepsilon^{-b-1}}{a\,\rho_1(m-1)}\eta\right] \eta-C^{p-1}\zeta^p\,\ge 0\,.
\end{equation}
Then $\bar u$ defined in \eqref{eq42} is a supersolution of equation \eqref{eq41}.
\end{proposition}

\begin{proof}[Proof of Proposition \ref{prop41}]
Let $\overline u$ be as in \eqref{eq41a} and $A$ as in \eqref{A}. In view of \eqref{eq43}, \eqref{eq44}, for any $(x,t)\in A$,
\begin{equation}\label{eq410}
\begin{aligned}
u_t-\frac{1}{\rho}({u}^m)_{xx}-{u}^p&= C\,F^{\frac{1}{m-1}-1}\left\{F\left[\zeta'\,+\frac{\zeta}{m-1}\,\frac{\eta'}{\eta}\right.\right.\\
&\left.+\frac{(R-|x|)^{-b-2}}{\rho}\,\frac{C^m-1}{a} \,\zeta^m\frac{m}{m-1}b\,\left(\frac{bm}{m-1}+1\right)\,\eta\right]\\
&-\frac{\zeta}{m-1}\,\frac{\eta'}{\eta}+\frac{(R-|x|)^{-b-2}}{\rho}\frac{C^m-1}{a} \,\zeta^m\frac{m}{(m-1)^2} b^2\,\eta\,\\
&\left.- C^{p-1}\,\zeta^p\,F^{\frac{p+m-2}{m-1}}\right\}.
\end{aligned}
\end{equation}
Due to assumption \eqref{h2}, for any $R-\varepsilon<|x|\le R$, we have
\begin{equation}\label{eq411}
\frac{(R-|x|)^{-b-2}}{\rho}\,\ge\,\frac{1}{c_1},\quad\quad -\frac{(R-|x|)^{-b-2}}{\rho}\,\ge\, \frac{1}{c_2}.
\end{equation}
By combining \eqref{eq410} and \eqref{eq411} we obtain 
\begin{equation}\label{eq414}
\begin{aligned}
&{u}_t-\frac{1}{\rho}({u}^m)_{xx}-{u}^p\\
&\ge C\,F^{\frac{1}{m-1}-1}\left\{F\left[\zeta'\,+\frac{\zeta}{m-1}\,\frac{\eta'}{\eta}+\frac{C^m-1}{a} \,\zeta^m\frac{m}{m-1}\frac{b}{c_2}\,\left(\frac{bm}{m-1}+1\right)\eta\right]\right.\\
&\left. -\frac{\zeta}{m-1}\,\frac{\eta'}{\eta}+\frac{C^m-1}{a} \,\zeta^m\frac{m}{(m-1)^2} \frac{b^2}{c_1}\,\eta- C^{p-1}\,\zeta^p\,F^{\frac{p+m-2}{m-1}}\right\}.
\end{aligned}
\end{equation}
By using the functions $\bar{\sigma}$, $\bar{\lambda}$, $\bar{\gamma}$ introduced in \eqref{eq47}, \eqref{eq414} reads
$$
{u}_t-\frac{1}{\rho}({u}^m)_{xx}-{u}^p\ge  C\,F^{\frac{1}{m-1}-1}\left\{\bar{\sigma}(t)F-\bar{\lambda}(t)-\bar{\gamma}(t)F^{\frac{p+m-2}{m-1}}\right\}.
$$
For each $t>0$, define
$$
\varphi(F):=\bar{\sigma}(t)F-\bar{\lambda}(t)-\bar{\gamma}(t)F^{\frac{p+m-2}{m-1}}, \quad F\in(0,1).
$$
Then ${u}$ is a supersolution of 
$$
{u}_t-\frac{1}{\rho}({u}^m)_{xx}-{u}^p=0\quad \text{for any}\,\,\, (x,t)\in A,
$$
if and only if, for each $t>0$,
$$
\varphi(F)\,\ge\,0 \quad \text{for any}\,\,\,0<F<1.
$$
Observe that $\varphi(F)$ is concave in the variable $F$, hence the latter reduces to the system
\begin{equation}\label{eq415}
\begin{cases}
&\varphi(0)\,\ge\,0\\
&\varphi(1)\,\ge\,0,
\end{cases}
\end{equation}
for each $t>0$ which is equivalent to
$$
\begin{cases}
&-\bar{\lambda}(t)\,\ge\,0\\
&\bar{\sigma}(t)-\bar{\lambda}(t)-\bar{\gamma}(t)\,\ge\,0.
\end{cases}
$$
The latter system is guaranteed by \eqref{eq48}, \eqref{eq49} and \eqref{kbound}. Thus, we have proved that
$$
{u}_t-\frac{1}{\rho}({u}^m)_{xx}-{u}^p\,\ge\,0 \quad \text{in}\,\,\, A.
$$
Since $u^m\in C^1([(-R,R)\setminus(-R+\varepsilon,R-\varepsilon)]\times(0,+\infty))$, in view of Lemma \ref{lemext}-(i) (applied with $I_1\times(0,+\infty)=A$, $I_2\times(0,+\infty)=[(-R,R)\setminus(-R+\varepsilon,R-\varepsilon)]\times(0,+ \infty)\setminus A$, $u_1= u$, $u_2=0$, $u=\bar u$), we can deduce that $\bar u$ is a supersolution of equation 
$$
{u}_t-\frac{1}{\rho}({u}^m)_{xx}-{u}^p=0 \quad \text{in}\,\,[(-R,R)\setminus(-R+\varepsilon,R-\varepsilon)]\times(0,+ \infty),
$$
in the sense of Definition \ref{def1}. Now, observe that, due to \eqref{eq22} and \eqref{eq48a}
$$
0<G<1\quad \text{for all}\,\,(x,t)\in (-R+\varepsilon,R-\varepsilon)\times(0,+\infty).
$$
By arguing as before, in view of \eqref{eq25}, \eqref{eq45}, \eqref{eq46}, \eqref{eq47}, for any, $(x,t)\in (-R+\varepsilon,R-\varepsilon)\times(0,+\infty)$, we have
\begin{equation}\label{eq416}
{v}_t-\frac{1}{\rho}({v}^m)_{xx}-{v}^p\ge C\,G^{\frac{1}{m-1}-1}\left\{\overline \sigma_0(t)G-\overline\lambda_0(t)-\overline\gamma(t)G^{\frac{p+m-2}{m-1}}\right\}.
\end{equation}
For each $t>0$, set
$$
\psi(G):=\overline \sigma_0(t)G-\overline\lambda_0(t)-\overline\gamma(t)G^{\frac{p+m-2}{m-1}}, \quad G\in(0,1).
$$
Our goal is to verify that, for each $t>0$,
$$
\psi(G)\ge0 \quad \text{for any}\,\,G\in(0,1).
$$
Similarly to the latter computation, we obtain the conditions
$$
\begin{cases}
&-\bar{\lambda}_0(t)\,\ge\,0\\
&\bar{\sigma}_0(t)-\bar{\lambda}_0(t)-\bar{\gamma}(t)\,\ge\,0,
\end{cases}
$$
which are guaranteed by \eqref{kbound}, \eqref{eq49a} and \eqref{eq49b}. Hence we have proved that
$$
{v}_t-\frac{1}{\rho}({v}^m)_{xx}-{v}^p\ge0\quad \text{for all}\,\,(x,t)\in (-R+\varepsilon,R-\varepsilon)\times(0,+\infty).
$$
Now observe that $\bar u\in C(D)$ and $\bar u^m\in C^1(D)$.
In conclusion, we can apply Lemma \ref{lemext}-(i) with $I_1\times(0,+\infty)=D\setminus[(-R+\varepsilon,R-\varepsilon)\times(0,\infty)]$, $I_2\times(0,+\infty)=(-R+\varepsilon,R-\varepsilon)\times(0,+\infty)$, $u_1= u$, $u_2= v$, $u=\bar u$. Hence we deduce that $\bar u$ is a supersolution of equation
$$
{u}_t - \frac{1}{\rho}({u}^m)_{xx}-{u}^p = 0 \quad \text{in } \,\,D\,,
$$
in the sense of Definition \ref{def1}.

\end{proof}

\begin{remark}\label{rem1}
Let $p>m$. \newline
In Theorem \ref{teo1} the precise hypotheses on the parameters $C>0$, $a>0$ and $T>0$ are the following:
$$
\alpha, \beta\,\,\text{as in}\,\,\eqref{eq26};
$$
\begin{equation}\label{eq417}
T^{\beta}>\frac{\varepsilon^{-b}}{a}\max\left\{1,\frac{\rho_2}{\rho_1}\frac{b}{m-1}\frac{R-\varepsilon}{\varepsilon}\right\};
\end{equation}
\begin{equation}\label{eq418}
\frac{p-m}{p-1}\,\ge\,\frac{C^{m-1}}{a}\frac{m}{m-1}\frac{b^2}{c_1};
\end{equation}
\begin{equation}\label{eq419}
\frac{C^{m-1}}{a}\frac{m}{m-1}b\left[\frac{1}{c_2}\left(\frac{bm}{m-1}+1\right)-\frac{b}{c_1(m-1)}\right]\,\ge\, C^{p-1}+\frac{1}{p-1};
\end{equation}
\begin{equation}\label{eq420}
\frac{p-m}{p-1}\,T^{\beta}\,\ge\,\frac{C^{m-1}}{a}\frac{m}{m-1}\frac{b^2}{\rho_1}\frac{\varepsilon^{-2b-2}}{a};
\end{equation}
\begin{equation}\label{eq421}
\frac{C^{m-1}}{a}\frac{m}{m-1}b\,\varepsilon^{-b-1}\left[\frac{1}{\rho_2(R-\varepsilon)}-\frac{b\,\varepsilon^{-b-1}}{a\,\rho_1(m-1)}T^{-\beta}\right]\,\ge\, C^{p-1}+\frac{1}{p-1}.
\end{equation}
\end{remark}

\begin{lemma}\label{lem1}
Conditions \eqref{eq417}- \eqref{eq421} in Remark \ref{rem1} can be satisfied simultaneously.
\end{lemma}
\begin{proof}
Let $\omega:=\frac{C^{m-1}}{a}$. Since $p>m$ the left-hand sides of \eqref{eq418} and \eqref{eq420} are positive. In view of \eqref{kbound} we can fix $\omega>0$ such that \eqref{eq418} and \eqref{eq420} hold and
$$
\begin{aligned}
\omega\,\min&\left\{\frac{m}{m-1}b\left[\frac{1}{c_2}\left(\frac{bm}{m-1}+1\right)-\frac{b}{c_1(m-1)}\right];\right.\\&\left.\frac{m}{m-1}b\left[\frac{\varepsilon^{-b-1}}{\rho_2(R-\varepsilon)}-\frac{b\,\varepsilon^{-2b-2}}{a\,\rho_1(m-1)}T^{-\beta}\right]\right\}\,>\, \frac{1}{p-1}.
\end{aligned}
$$
Finally, we can take $C>0$ sufficiently small that \eqref{eq419} and \eqref{eq420} hold and so $a>0$ is accordingly fixed.
\end{proof}

\begin{proof}[Proof of Theorem \ref{teo1}]
We prove Theorem \ref{teo1} by means of Proposition \ref{prop41}. Due to Lemma \eqref{lem1} we can assume than condition \eqref{eq417}-\eqref{eq421} hold. Set
$$
\zeta(t)=(T+t)^{-\alpha}, \quad \eta(t)=(T+t)^{-\beta}\quad \text{for all}\,\,\,t>0,
$$
Then conditions \eqref{eq48}, \eqref{eq49}, \eqref{eq49a} and \eqref{eq49b} read
\begin{equation}\label{eq422}
\beta(T+t)^{\beta-1}\ge\frac{C^{m-1}}{a}\frac{m}{m-1}(T+t)^{-\alpha(m-1)}\frac{b^2}{c_1};
\end{equation}
\begin{equation}\label{eq423}
\begin{aligned}
-\alpha(T+t)^{-\alpha -1}&+\frac{C^{m-1}}{a}\frac{m\,b}{m-1}\left[\frac{1}{c_2}\left(\frac{bm}{m-1}+1\right)-\frac{b}{c_1(m-1)}\right](T+t)^{-\alpha\,m-\beta}\\&-C^{p-1}(T+t)^{-\alpha\,p}\ge 0;
\end{aligned}
\end{equation}
\begin{equation}\label{eq422b}
\beta(T+t)^{2\beta-1}\ge\frac{C^{m-1}}{a}\frac{m}{m-1}(T+t)^{-\alpha(m-1)}\frac{b^2}{\rho_1}\frac{\varepsilon^{-2b-2}}{a};
\end{equation}
\begin{equation}\label{eq424}
\begin{aligned}
-\alpha(T+t)^{-\alpha -1}&+\frac{C^{m-1}}{a}\frac{m\,b}{m-1}\left[\frac{\varepsilon^{-b-1}}{\rho_2(R-\varepsilon)}-\frac{b\,\varepsilon^{-2b-2}}{a\,\rho_1(m-1)}(T+t)^{-\beta}\right]\\&\cdot(T+t)^{-\alpha\,m-\beta}-C^{p-1}(T+t)^{-\alpha\,p}\ge 0.
\end{aligned}
\end{equation}
By the definition of $\alpha$ and $\beta$ given in equation \eqref{eq26}, \eqref{eq422}, \eqref{eq423}, \eqref{eq422b} \eqref{eq424} become
\begin{equation*}
\frac{p-m}{p-1}\,\ge\,\frac{C^{m-1}}{a}\frac{m}{m-1}\frac{b^2}{c_1};
\end{equation*}
\begin{equation*}
\frac{C^{m-1}}{a}\frac{m}{m-1}b\left[\frac{1}{c_2}\left(\frac{bm}{m-1}+1\right)-\frac{b}{c_1(m-1)}\right]\,\ge\, C^{p-1}+\frac{1}{p-1};
\end{equation*}
\begin{equation*}
\frac{p-m}{p-1}\,T^{\beta}\,\ge\,\frac{C^{m-1}}{a}\frac{m}{m-1}\frac{b^2}{\rho_1}\frac{\varepsilon^{-2b-2}}{a};
\end{equation*}
\begin{equation*}
\frac{C^{m-1}}{a}\frac{m}{m-1}b\,\varepsilon^{-b-1}\left[\frac{1}{\rho_2(R-\varepsilon)}-\frac{b\,\varepsilon^{-b-1}}{a\,\rho_1(m-1)}T^{-\beta}\right]\,\ge\, C^{p-1}+\frac{1}{p-1}.
\end{equation*}
Therefore, \eqref{eq48}, \eqref{eq49}, \eqref{eq49a} and \eqref{eq49} follow from conditions \eqref{eq418}, \eqref{eq419}, \eqref{eq420} and \eqref{eq421}. Furthermore, condition \eqref{eq48a} is verified due to assumption \eqref{eq417}. The conclusion follows from Propositions \ref{prop41} and \ref{cpsup}.

\end{proof}

\subsection{Proof of blow-up for the fast decaying rate}\label{Blowup}

We want to construct a suitable family of subsolution of equation
\begin{equation}\label{eq51}
u_t -\frac{1}{\rho(x)}(u^m)_{xx}-u^p=0 \quad \text{ in } (-R,R)\times[0,T)=:D_T.
\end{equation}
For this purpose, we set, 
\begin{equation}\label{eq52}
\underline{u}(x,t)\equiv w(x,t) \quad \text{for all}\,\, (x,t)\in D_T.
\end{equation}
where $w$ has been defined in \eqref{eq42b}, \eqref{eq42}, \eqref{eq42a} and \eqref{eq42c} with $\zeta,\eta\in C^1([0,T);[0,+\infty))$. Furthermore, we assume that $\varepsilon$ in \eqref{eq42c} is such that 
\begin{equation}\label{eps}
0<\varepsilon<\frac{b}{b+2}R,
\end{equation}
and then we set, without lost of generality, $\varepsilon_0=\varepsilon$ in \eqref{h2}. Observe that consequently $c_1$ and $c_2$ will chance accordingly. Now, we define
\begin{equation}\label{BB}
\begin{aligned}
&S_1:=\{(x,t)\in [(-R,-R+\varepsilon)\cup (R-\varepsilon,R)]\times(0,T):0<F(x,t)<1\};\\
&S_2:=\{(x,t)\in (-R+\varepsilon, R-\varepsilon)\times(0,T):0<G(x,t)<1\};
\end{aligned}
\end{equation}
with $F$, $G$ as in \eqref{eq42c} and 
\begin{equation}\label{eq510}
\begin{aligned}
& \underline{\sigma}(t) := \zeta ' +  \frac{\zeta}{m-1} \frac{\eta'}{\eta} + \frac{C^{m-1}}{a}\, \zeta^m\,\frac{m}{m-1}\frac{b}{c_1}\left(\frac{b\,m}{m-1}+1\right)\eta;\\
& \underline{\lambda}(t) := \frac{\zeta}{m-1} \frac{\eta'}{\eta} +  \frac{C^{m-1}}{a}\,\zeta^m\,\frac{m}{(m-1)^2}\,\frac{b^2}{c_2}\,\eta; \\
& \underline{\gamma}(t):=C^{p-1} \zeta^p;\\
& \underline{\sigma}_0(t) := \zeta ' +  \frac{\zeta}{m-1} \frac{\eta'}{\eta} + \frac{C^{m-1}}{a}\, \zeta^m\,\frac{m(m+1)}{(m-1)^2}\,\frac{b}{\rho_1}\,\frac{\varepsilon^{-b-1}}{R-\varepsilon}\,\eta;\\
& \underline{\lambda}_0(t) := \frac{\zeta}{m-1} \frac{\eta'}{\eta} +  \frac{C^{m-1}}{a}\,\zeta^m\,\frac{m}{(m-1)^2}\,\,\frac{2b}{\rho_2}\,\frac{\varepsilon^{-b-1}}{R-\varepsilon}\,\eta; \\
& \textit{K}:=  \left (\frac{m-1}{p+m-2}\right)^{\frac{m-1}{p-1}} - \left (\frac{m-1}{p+m-2}\right)^{\frac{p+m-2}{p-1}} >0.
\end{aligned}
\end{equation}

\begin{proposition}\label{prop2}
Let $T\in (0,\infty)$, $\zeta$, $\eta \in C^1([0,T);[0, +\infty))$. Assume that \eqref{eps} holds and let $\underline{\sigma},\underline{\lambda},\underline{\gamma},\underline{\sigma}_0,\underline{\lambda}_0, \textit{K}$ be defined in \eqref{eq510}. Assume that, for all $t\in(0,T)$,
\begin{equation}\label{eq511}
\textit{K}\,[\underline{\sigma}(t)]^{\frac{p+m-2}{p-1}} \le \underline{\lambda}(t) [\underline{\gamma}(t)]^{\frac{m-1}{p-1}}, \end{equation}
\begin{equation}\label{eq512}
(m-1) \underline{\sigma}(t) \le (p+m-2) \underline{\gamma}(t)\,,
\end{equation}
\begin{equation}\label{eq513}
\textit{K}\,[\underline{\sigma}_0(t)]^{\frac{p+m-2}{p-1}} \le \underline{\lambda}_0(t) [\underline{\gamma}(t)]^{\frac{m-1}{p-1}}, \end{equation}
\begin{equation}\label{eq514}
(m-1) \underline{\sigma}_0(t) \le (p+m-2) \underline{\gamma}(t)\,.
\end{equation}

Then $\underline u$ defined in \eqref{eq52} is a weak subsolution of equation \eqref{eq51}.
\end{proposition}

\begin{proof}[Proof of Proposition \ref{prop2}]
Let $\underline u$ be as in \eqref{eq52} and $S_1$, $S_2$ as in \eqref{BB}. In view of \eqref{eq43} and \eqref{eq44}, for all $(x,t)\in B_1$ we have
\begin{equation}\label{eq515}
\begin{aligned}
&u_t-\frac{1}{\rho}(u^m)_{xx}-u^p \\
&= CF^{\frac{1}{m-1}-1}\left\{F\left[\zeta'+\frac{\zeta}{m-1}\frac{\eta'}{\eta}+\frac{(R-|x|)^{-b-2}}{\rho(x)}\frac{C^{m-1}}{a}\frac{b\,m\, \zeta^m}{m-1}\left(\frac{bm}{m-1}+1\right)\eta\right]\right.\\
&\left.-\frac{\zeta}{m-1}\,\frac{\eta'}{\eta}+\frac{(R-|x|)^{-b-2}}{\rho(x)}\frac{C^{m-1}}{a} \frac{b^2\,m\,\zeta^m}{(m-1)^2} \eta- C^{p-1}\,\zeta^p\,F^{\frac{p+m-2}{m-1}}\right\}.
\end{aligned}
\end{equation}
Due to assumption \eqref{h2}, for any $R-\varepsilon<|x|\le R$, we have
\begin{equation}\label{eq516}
\frac{(R-|x|)^{-b-2}}{\rho}\,\le\,\frac{1}{c_2},\quad -\frac{(R-|x|)^{-b-2}}{\rho}\,\le\, -\frac{1}{c_1}.
\end{equation}
By combining \eqref{eq515} and \eqref{eq516} we obtain 
\begin{equation}\label{eq517}
\begin{aligned}
{u}_t&-\frac{1}{\rho}({u}^m)_{xx}-{u}^p\\
&\le C\,F^{\frac{1}{m-1}-1}\left\{F\left[\zeta'\,+\frac{\zeta}{m-1}\,\frac{\eta'}{\eta}+\frac{C^m-1}{a} \,\zeta^m\frac{m}{m-1}\frac{b}{c_1}\,\left(\frac{bm}{m-1}+1\right)\eta\right]\right.\\
&\left. -\frac{\zeta}{m-1}\,\frac{\eta'}{\eta}+\frac{C^m-1}{a} \,\zeta^m\frac{m}{(m-1)^2} \frac{b^2}{c_2}\,\eta- C^{p-1}\,\zeta^p\,F^{\frac{p+m-2}{m-1}}\right\}.
\end{aligned}
\end{equation}
By using the functions $\underline{\sigma}$, $\underline{\lambda}$, $\underline{\gamma}$ introduced in \eqref{eq510}, \eqref{eq517} becomes
$$
{u}_t-\frac{1}{\rho}({u}^m)_{xx}-{u}^p\le  C\,F^{\frac{1}{m-1}-1}\left\{\underline{\sigma}(t)F-\underline{\lambda}(t)-\underline{\gamma}(t)F^{\frac{p+m-2}{m-1}}\right\}.
$$
We set, for each $t\in (0, T)$,
$$\varphi(F):=\underline{\sigma}(t)F - \underline{\lambda}(t) - \underline{\gamma}(t)F^{\frac{p+m-2}{m-1}}.$$
Our goal is to find suitable $C,a,\zeta, \eta$ such that, for each $t\in (0, T)$,
$$\varphi(F) \le 0 \quad \textrm{for any}\,\,\, F \in (0,1)\,.$$
To this aim, we impose that
$$\sup_{F\in (0,1)}\varphi(F)=\max_{F\in (0,1)}\varphi(F)= \varphi (F_0)\leq 0\,,$$
for some $F_0\in (0,1).$
We have
$$ \begin{aligned} \frac{d \varphi}{dF}=0 &\iff \underline{\sigma}(t) - \frac{p+m-2}{m-1} \underline{\gamma}(t) F^{\frac{p-1}{m-1}} =0 \\ & \iff F=F_0= \left [\frac{m-1}{p+m-2} \frac{\underline{\sigma}(t)}{\underline{\gamma}(t)} \right ]^{\frac{m-1}{p-1}} \,.\end{aligned}$$
Then
$$ \varphi(F_0)= K\, \frac{\underline{\sigma}(t)^{\frac{p+m-2}{p-1}}}{\underline{\gamma}(t)^{\frac{m-1}{p-1}}} - \underline{\lambda}(t)\,, $$
with $K=K(m,p)$ as in \eqref{eq510}. By hypoteses \eqref{eq511} and \eqref{eq512}, for each $t\in (0, T)$,
\begin{equation}\label{eq518}
\varphi(F_0) \le 0\,,\quad
F_0 \le 1\,.
\end{equation}
So far, we have proved that
\begin{equation}\label{eq519}
u_t-\frac{1}{\rho(x)}(u^m)_{xx}-u^p \le 0 \quad \text{ in }\,\, S_1.
\end{equation}
Furthermore, since $u^m\in C^1([(-R,R)\setminus (-R+\varepsilon,R-\varepsilon)]\times (0, T))$, due to Lemma \ref{lemext} (applied with $I_1\times(0,T)=S_1, I_2\times(0,T)=[[(-R,R)\setminus (-R+\varepsilon,R-\varepsilon)]\times(0,T)]\setminus S_1, u_1=u, u_2=0, u=u$), it follows that $u$ is a subsolution to equation
\[u_t-\frac{1}{\rho(x)}(u^m)_{xx}-u^p = 0 \quad \text{ in }\,\, [(-R,R)\setminus (-R+\varepsilon,R-\varepsilon)]\times (0, T),\]
in the sense of Definition \ref{def1}.

Using \eqref{eq21}, \eqref{eps}, \eqref{eq45} and \eqref{eq46} yields, for all $(x,t)\in S_2$,
\begin{equation}\label{eq520}
\begin{aligned}
&v_t- \frac{1}{\rho}(v^m)_{xx} - v^p  \\
&\le C G^{\frac{1}{m-1}-1}\left\{G\left[\zeta' + \frac{\zeta}{m-1}\frac{\eta'}{\eta}+ \frac{C^{m-1}}{a}\zeta^m\frac{m(m+1)}{(m-1)^2}\frac{b}{\rho_1}\frac{\varepsilon^{-b-1}}{R-\varepsilon}\eta \right]\right.\\
&\left.-\frac{\zeta}{m-1}\frac{\eta'}{\eta} - \frac{C^{m-1}}{a} \zeta^m\,\frac{m}{(m-1)^2}\frac{2b}{R-\varepsilon}\,\frac{\varepsilon^{b+1}}{\rho_2}\eta-C^{p-1}\zeta^p G^{\frac{p+m-2}{m-1}} \right\}.\\
\end{aligned}
\end{equation}
Now, by the same arguments used to obtain \eqref{eq519}, in view of \eqref{eq513} and \eqref{eq514} we can infer that
\begin{equation}\label{eq521}
v_t  - \frac 1{\rho} (v^m)_{xx} -v^p\le 0\quad \textrm{for any}\,\,\, (x,t)\in S_2\,.
\end{equation}
Moreover, since $v^m\in C^1((-R+\varepsilon,R-\varepsilon)\times (0, T))$, in view of Lemma \ref{lemext} (applied with $I_1\times(0,T)=S_2, I_2\times(0,T)=[(-R+\varepsilon,R-\varepsilon)\times(0,T)]\setminus S_2, u_1= v, u_2=0, u=v$), we get that
$v$ is a subsolution to equation
\begin{equation}\label{eq522}
v_t  - \frac 1{\rho} (v^m)_{xx} -v^p=0\quad \textrm{in}\,\,\, (-R+\varepsilon,R-\varepsilon)\times (0, T)\,,
\end{equation}
in the sense of Definition \ref{def1}. Now, observe that $ \underline u \in C(D_T)$; indeed,
$$
u = v  = C \zeta \left [ 1- \frac{\varepsilon^{-b}}{a} \,\eta\right ]_+^{\frac{1}{m-1}} \quad \textrm{in}\,\, \{-R+\varepsilon,R-\varepsilon\}\times (0, T)\,.
$$
Moreover, 
\begin{equation}\label{eq523}
\begin{aligned}
&(u^m)_x = (v^m)_x  \\
&= -C^m \zeta^m \frac{m\,b\,\varepsilon^{-b-1}}{m-1} \frac{\eta(t)}{a} \left [ 1- \frac{\varepsilon^{-b}}{a} \eta\right ]_+^{\frac{1}{m-1}} \quad \textrm{in}\,\, \{-R+\varepsilon,R-\varepsilon\}\times (0, T)\,.
\end{aligned}
\end{equation}
In conclusion, in view of \eqref{eq523} and Lemma \ref{lemext} (applied with $I_1\times(0,T)=D_T\setminus[(-R+\varepsilon,R-\varepsilon)\times(0,T)], I_2\times(0,T)=(-R+\varepsilon,R-\varepsilon)\times(0,T), u_1= u, u_2=v, u=\underline  u$), we can infer that $\underline u$ is a subsolution to equation \eqref{eq51}, in the sense of Definition \ref{def1}.
\end{proof}

\begin{remark}\label{rem2} Let $\omega:=\frac{C^{m-1}}{a}.$ In Theorem \ref{teo2}
the precise choice of the parameters $C>0, a>0, T>0$ are the following.
\begin{itemize}
\item[(a)] Let $p>m$. We require that
\begin{equation}\label{eq524}
\begin{aligned}
K &\left [\frac{1}{m-1} + \omega \frac{m}{m-1}\,\frac{b}{c_1}\left(\frac{bm}{m-1}+1\right) \right]^{\frac{p+m-2}{p-1}}\\& \le \frac{C^{m-1}}{m-1} \left [ \omega\,\frac{m}{m-1} \frac{b^2}{c_2}  + \frac{p-m}{p-1} \right ] \,;
\end{aligned}
\end{equation}
\begin{equation}\label{eq525}
1+ \omega\,\frac{b\,m}{c_1}  \left(\frac{bm}{m-1}+1\right)  \le \left (p+m-2 \right )C^{p-1}\,;
\end{equation}
\begin{equation}\label{eq526}
\begin{aligned}
&K\left[ \frac1{m-1} + \omega \,\frac{\varepsilon^{-b-1}}{\rho_1}\, \frac{m(m+1)}{(m-1)^2}\frac{b}{R-\varepsilon} \right]^{\frac{p+m-2}{p-1}}\\&\leq \frac{C^{m-1}}{m-1}\left[\omega\frac{\varepsilon^{-b-1}}{\rho_2}\frac m{m-1} \frac{2b}{R-\varepsilon}+ \frac{p-m}{p-1} \right] \,;
\end{aligned}
\end{equation}
\begin{equation}\label{eq527}
1+ \omega\,\frac{1}{\rho_1}\,\frac{m(m+1)}{m-1}\varepsilon^{-b-1}\frac{b}{R-\varepsilon}\leq (p+m-2) C^{p-1}\,.
\end{equation}
\item[(b)] Let $p<m$. We require that
\begin{equation}\label{eq528}\begin{aligned}
& \omega > \frac{(m-p)(m-1)}{b(p-1)m} \max\left\{\frac{c_2}{b},\, \frac{(R-\varepsilon)\rho_2}{2\varepsilon^{-b-1}}     \right\} ;\\
\end{aligned}
\end{equation}
\begin{equation}\label{eq529}
\begin{aligned}
a \ge \frac{K(m-1)}{\omega}\max&\left\{\frac{ \left [\frac{1}{m-1} + \omega \frac{m}{m-1}\,\frac{b}{c_1}\left(\frac{bm}{m-1}+1\right) \right]^{\frac{p+m-2}{p-1}}}
{ \omega\,\frac{m}{m-1} \frac{b^2}{c_2}  + \frac{p-m}{p-1} },\right. \\
& \left.  \frac{\left[ \frac1{m-1} + \omega \,\frac{\varepsilon^{-b-1}}{\rho_1}\, \frac{m(m+1)}{(m-1)^2}\frac{b}{R-\varepsilon} \right]^{\frac{p+m-2}{p-1}}}
{ \omega\frac{\varepsilon^{-b-1}}{\rho_2}\frac m{m-1} \frac{2b}{R-\varepsilon}+ \frac{p-m}{p-1}} \right\}\,;
\end{aligned}
\end{equation}
\begin{equation}\label{eq530}
\begin{aligned}
\left (p+m-2 \right ) \left (a\omega \right)^{\frac{p-1}{m-1}} \ge &  \max\left\{ 1+ \omega\,\frac{b\,m}{c_1}  \left(\frac{bm}{m-1}+1\right) , \right. \\ &\left. 1+ \omega\,\frac{1}{\rho_1}\,\frac{m(m+1)}{m-1}\varepsilon^{-b-1}\frac{b}{R-\varepsilon}\right\}\,.
\end{aligned}
\end{equation}

\item[(c)] Let $p=m$. We require that $\omega>0$,
\begin{equation}\label{eq531}
\begin{aligned}
a \ge &\frac{K(m-1)^2}{m\,b\,\omega^2}\max\left\{\frac{c_2}{b}\, \left [\frac{1}{m-1} + \omega \frac{m}{m-1}\,\frac{b}{c_1}\left(\frac{bm}{m-1}+1\right) \right]^{2},\right. \\
& \left. \frac{\rho_2} {\varepsilon^{-b-1}}\frac{R-\varepsilon}{2} \left[ \frac1{m-1} + \omega \,\frac{\varepsilon^{-b-1}}{\rho_1}\, \frac{m(m+1)}{(m-1)^2}\frac{b}{R-\varepsilon} \right]^{2} \right\}\,;
\end{aligned}
\end{equation}
\end{itemize}
\end{remark}

\begin{lemma}\label{lemma2}
All the conditions of Remark \ref{rem2} can hold simultaneously.
\end{lemma}
\begin{proof}
(a) We take any $\omega>0$, then we select $C>0$ big enough (therefore, $a>0$ is also fixed, due to the definition of $\omega$) so that \eqref{eq524}-\eqref{eq527} hold.

\smallskip

\noindent (b) We can take $\omega>0$ so that \eqref{eq528} holds, then we take $a>0$ sufficiently large to guarantee \eqref{eq529} and \eqref{eq530} (therefore, $C>0$ is also fixed).

\smallskip

\noindent (c) For any $\omega>0$, we take $a>0$ sufficiently large to guarantee \eqref{eq531} (thus, $C>0$ is also fixed).
\end{proof}

\begin{proof}[Proof of Theorem \ref{teo2}] We now prove Theorem \ref{teo2}, by means of Proposition \ref{prop2}. In view of Lemma \ref{lemma2}, we can assume that all the conditions in Remark \ref{rem2} are fulfilled.
Set
$$\zeta(t)=(T-t)^{-\alpha}\,, \quad \eta(t)=(T-t)^{\beta}$$ and $$\alpha=\frac{1}{p-1}\,,\quad \beta=\frac{m-p}{p-1}\,.$$ Then
\begin{equation*}
\underline{\sigma}(t) = \left [\frac{1}{m-1} + \frac{C^{m-1}}{a} \frac{m}{m-1} \frac{b}{c_1} \left (\frac{bm}{m-1}+1\right ) \right ] \left (T-t \right)^{-\frac{p}{p-1}}\,,
\end{equation*}
\begin{equation*}
\underline{\lambda}(t) :=  \left [\frac{m-p}{(m-1)(p-1)} + \frac{C^{m-1}}{a} \frac{m}{(m-1)^2} \frac{b^2}{c_2} \right ] \left(T-t\right)^{-\frac{p}{p-1}}\,,
\end{equation*}
\begin{equation*}
\underline{\gamma}(t):=C^{p-1} \left(T-t\right)^{-\frac{p}{p-1}}\,.
\end{equation*}

Let $p>m$. Conditions \eqref{eq524} and \eqref{eq525} imply \eqref{eq511} and \eqref{eq512}, whereas \eqref{eq526} and \eqref{eq527} imply \eqref{eq513} and \eqref{eq514}. Hence, by Propositions \ref{prop2} and \ref{cpsub} the thesis follows in this case.

Let $p<m$. Conditions \eqref{eq528}, \eqref{eq529} and \eqref{eq530} imply \eqref{eq511}, \eqref{eq512}, \eqref{eq513} and \eqref{eq514}. Hence, by Propositions \ref{prop2} and \ref{cpsub}  the thesis follows in this case, too.

Finally, let $p=m$. Condition \eqref{eq531} implies \eqref{eq511}, \eqref{eq512}, \eqref{eq513} and \eqref{eq514}. Hence, by Propositions \ref{prop2} and \ref{cpsub}  the thesis follows in this case, too. The proof is complete.

\end{proof}

\medskip

We can now prove Theorem \ref{teo6} by means of Theorem \ref{teo2}.
\begin{proof}[Proof of Theorem \ref{teo6}]
Since $u_0\not\equiv 0$ and $u_0\in C((-R,R))$, there exists $\nu>0$ and $I\subset (-R,R)$ such that
 $$u_0(x) \ge \nu, \quad \textrm{for all}\,\,\, x \in I.$$
Without loss of generality, we can assume that $I=(-r_0,r_0)$, for some $r_0>0$. Let $\underline u$ be the subsolution of problem \eqref{problema} considered in Theorem \ref{teo2} (with $a>0$ and $C>0$ properly fixed). We can find $T>0$ sufficiently big such that
\begin{equation}
\begin{aligned}
&C\, T^{-\frac{1}{p-1}} \le \nu, \\  
&a\, T^{-\frac{m-p}{p-1}} \le \min\left\{(R-r_0)^{-b}, \frac{b\,r_0^2}{2\varepsilon^{b+1}(R-\varepsilon)}+\frac{2\varepsilon-b(R-\varepsilon)}{2\varepsilon^{b+1}}\right\}.
\end{aligned}
\label{condcorollario}
\end{equation}
From inequalities in \eqref{condcorollario}, we can deduce that
\begin{equation*}
\underline u(x,0) \le u_0(x) \quad \textrm{for any}\,\,\, x\in (-R,R).
\end{equation*}
Hence, by Theorem \ref{teo2} and the comparison principle, the thesis follows.
\end{proof}

\section{Proofs of Theorems \ref{teo3} and \ref{teo4}}\label{critical}

\subsection{Proof of global existence for the critical decaying rate}
We assume \eqref{h1} and \eqref{h2} with $q=2$. We aim to construct a suitable family of supersolutions to
\begin{equation}\label{eq70bc}
u_t -\frac{1}{\rho(x)}(u^m)_{xx}-u^p=0 \quad \text{ in } (-R,R)\times[0,+\infty)=:D,
\end{equation}
Hence, we define, for all $(x,t)\in D$,
\begin{equation}
{\bar{u}}(x,t):=C\zeta(t)\left [1-\frac{(R-|x|)^{-\delta}}{a}\eta(t)\right]_{+}^{\frac{1}{m-1}}\,,
\label{eq731}
\end{equation}
for some $\delta>0$, $\eta$, $\zeta \in C^1([0, +\infty); [0, +\infty))$ and $C > 0$, $a > 0$.
\smallskip

Let us set
$$
F(x,t):= 1-\frac{(R-|x|)^{-\delta}}{a}\,\eta(t)\,,
$$
and define
\begin{equation}\label{eq731b}
A:=\left \{ (x,t) \in D \setminus [\{0\}  \times (0,+\infty)]\,\, | \,\,0<F(x,t)<1 \right \}.
\end{equation}
For any $(x,t) \in A$, we have:

\begin{equation}
\bar{u}_t =C\zeta ' F^{\frac{1}{m-1}} + C\zeta \frac{1}{m-1} \frac{\eta'}{\eta}F^{\frac{1}{m-1}} - C\zeta \frac{1}{m-1} \frac{\eta'}{\eta} F^{\frac{1}{m-1}-1}. \\
\label{eq732}
\end{equation}
\begin{equation}
\begin{aligned}
(\bar{u}^m)_{xx}&=-\delta\,\frac{C^m}{a^2} \zeta^m \frac{m}{(m-1)^2}(R-|x|)^{-2\delta-2}\,\eta^2 F^{\frac{1}{m-1}-1}  \\ 
&- \delta(\delta+1)\frac{C^m}{a} \zeta^m \frac{m}{m-1} (R-|x|)^{-\delta-2}\,\eta F^{\frac{1}{m-1}} .
\end{aligned}
\label{eq733}
\end{equation}

We also define
\begin{equation}\label{eq734}
\begin{aligned}
&\bar{\sigma}(t) := \zeta ' + \zeta \frac{1}{m-1} \frac{\eta'}{\eta} + \frac{\delta(\delta+1)}{c_2} \frac{C^{m-1}}{a} R^{-\delta}\,\zeta^m \frac{m}{m-1} \eta ,\\
&\bar{\lambda}(t) := \zeta \frac{1}{m-1} \frac{\eta'}{\eta} \\%- \frac{\delta}{c_2} \frac{C^{m-1}}{a^2} R^{-2\delta}\,\zeta^m \frac{m}{(m-1)^2} \eta^2 ,\\
& \bar{\gamma}(t):=C^{p-1}\zeta^p\,.
\end{aligned}
\end{equation}

\begin{proposition}\label{prop72}
Let $\zeta=\zeta(t)$, $\eta=\eta(t) \in C^1([0,+\infty);[0, +\infty))$. Let $\bar\sigma$, $\bar\lambda$, $\bar\gamma$ be as defined in \eqref{eq734}. Assume \eqref{h1}, \eqref{h2} with $q=2$ and that, for all $t\in (0,+\infty)$,
\begin{equation}
\eta'\le 0
\label{eq735}
\end{equation}
and
\begin{equation}
\zeta' + \frac{\delta(\delta+1)}{c_2} \frac{C^{m-1}}{a} R^{-\delta}\,\zeta^m \frac{m}{m-1} \eta  -C^{p-1}\zeta^p \ge 0\,.%\left[\delta+1+\frac{R^{-\delta}}{m-1}\frac{\eta}{a}\right]  -C^{p-1}\zeta^p \ge 0\,.
\label{eq736}
\end{equation}
then $\bar{u}$ defined in \eqref{eq731} is a supersolution of equation \eqref{eq70bc}.
\end{proposition}

\begin{proof}[Proof of Proposition \ref{prop72}]

Let $\overline u$ be as in \eqref{eq731} and $A$ as in \eqref{eq731b}. In view of \eqref{eq732} and \eqref{eq733}, for any $(x,t)\in S$,
\begin{equation}\label{eq735}
\begin{aligned}
\bar{u}_t - \frac{1}{\rho}(\bar{u}^m)_{xx}-\bar{u}^p 
&=C\zeta ' F^{\frac{1}{m-1}} + C\zeta \frac{1}{m-1} \frac{\eta'}{\eta}F^{\frac{1}{m-1}} - C\zeta \frac{1}{m-1} \frac{\eta'}{\eta} F^{\frac{1}{m-1}-1}\\ 
&+\frac{\delta}{\rho}\,\frac{C^m}{a^2} \zeta^m \frac{m}{(m-1)^2}(R-|x|)^{-2\delta-2}\,\eta^2 F^{\frac{1}{m-1}-1} \\ 
& +\frac{ \delta(\delta+1)}{\rho}\frac{C^m}{a} \zeta^m \frac{m}{m-1} (R-|x|)^{-\delta-2}\,\eta F^{\frac{1}{m-1}} - C^p \zeta^p F^{\frac{p}{m-1}}.
\end{aligned}
\end{equation}
Thanks to hypothesis \eqref{h2} with $q=2$, we have
\begin{equation}\label{eq736}
\begin{aligned}
&\frac{(R-|x|)^{-2\delta-2}}{\rho}\ge \frac{(R-|x|)^{-2\delta-2+2}}{c_2}\ge 0, \quad \text{for all }\,\, x\in (-R,R)\\
& \frac{(R-|x|)^{-\delta-2}}{\rho} \ge \frac{(R-|x|)^{-\delta-2+2}}{c_2}\ge \frac{R^{-\delta}}{c_2} \quad \text{for all }\,\, x\in (-R,R)\,,
\end{aligned}
\end{equation}
From \eqref{eq735} and \eqref{eq736} we get,
\begin{equation}\label{eq737}
\begin{aligned}
&\bar{u}_t - \frac{1}{\rho}(\bar{u}^m)_{xx}-\bar{u}^p  \\
&\ge CF^{\frac{1}{m-1}-1} \left \{F\left [\zeta ' + \zeta \frac{1}{m-1} \frac{\eta'}{\eta} + \frac{\delta(\delta+1)}{c_2} \frac{C^{m-1}}{a} R^{-\delta}\,\zeta^m \frac{m}{m-1} \eta \right ] \right . \\
& \left .-\zeta \frac{1}{m-1} \frac{\eta'}{\eta} - C^{p-1} \zeta^p F^{\frac{p+m-2}{m-1}} \right \}
\end{aligned}
\end{equation}
From \eqref{eq737} and \eqref{eq734}, we have
\begin{equation}\label{eq738}
\bar u_t - \frac{1}{\rho}(\bar u^m)_{xx}-\bar u^p\geq C F^{\frac{1}{m-1}-1} \left[\bar{\sigma}(t)F - \bar{\lambda}(t) - \bar{\gamma}(t)F^{\frac{p+m-2}{m-1}}\right ]\,.
\end{equation}
For each $t>0$, set
$$\varphi(F):=\bar{\sigma}(t)F - \bar{\lambda}(t) - \bar{\gamma}(t)F^{\frac{p+m-2}{m-1}}, \quad F\in (0,1)\,.$$
By arguing as in the proof of Proposition \ref{prop41} we get, for any $t>0$, the conditions
\begin{equation}
\begin{aligned}
&\varphi(0) \ge 0\,,  \\
&\varphi(1) \ge 0  \,.
\end{aligned}
\end{equation}
These are equivalent to
$$
-\bar\lambda(t) \ge 0\,, \quad\quad \bar\sigma(t)-\bar\lambda(t)-\bar\gamma(t) \ge 0\,,
$$
that is
$$
-\frac{\zeta}{m-1}\frac{\eta'}{\eta} \ge 0, \quad\quad \zeta' +\frac{\delta(\delta+1)}{c_1} \frac{C^{m-1}}{a} R^{-\delta}\,\zeta^m \frac{m}{m-1} \eta -C^{p-1}\zeta^p \ge 0\,,
$$
which are guaranteed by \eqref{eq735} and \eqref{eq736}. Hence we have proved that
$$
\bar{u}_t - \frac{1}{\rho}(\bar{u}^m)_{xx}-\bar{u}^p \ge 0 \quad \text{in } \text{A}\,.
$$
Now observe that
\begin{itemize}
\item[]$\bar{u} \in C((-R,R) \times [0,+\infty))$\,,
\item[]$\bar{u}^m \in C^1([(-R,R) \setminus \{0\}] \times [0,+\infty))$\,, and by the definition of $\bar{u}$\,,
\item[]$\bar{u}\equiv0$ in $[((-R,R)\setminus\{0\})\times(0,+\infty)]\setminus {A}$\,.
\end{itemize}
Hence, by Lemma \ref{lemext} (applied with $I_1\times(0,+\infty)=A$, $I_2\times(0,+\infty)=[((-R,R)\setminus\{0\})\times(0,+\infty)]\setminus {A}$, $u_1=\bar u$, $u_2=0$, $u=\bar u$), $\bar u$ is a supersolution of equation
$$
\bar{u}_t - \frac{1}{\rho}(\bar{u}^m)_{xx}-\bar{u}^p = 0 \quad \text{in } ((-R,R)\setminus \{0\})\times (0,+\infty)
$$
in the sense of Definition \ref{def1}. Thanks to a Kato-type inequality,  since $$\frac{\partial \bar u^m}{\partial |x|}(0,t)\le 0,$$ we can infer that $\bar u$ is a supersolution of equation \eqref{eq70bc} in the sense of Definition \ref{def1}.
\end{proof}

\begin{remark}\label{rem72}
Let $$p>m,$$ and $\omega:=\frac{C^{m-1}}{a}$. In Theorem \ref{teo3} the precise hypothesis on parameters $C>0$, $\omega>0$, $T>0$ is the following:
\begin{equation}\label{eq739}
\omega \frac{\delta(\delta+1)}{c_2} \frac{m}{m-1}\, R^{-\delta} \ge C^{p-1} + \frac{1}{p-1}\,.
\end{equation}
\medskip

Observe that the latter hypothesis is guaranteed if one first fixes $\omega>0$, and then one takes $C>0$ so small that \eqref{eq739} holds (and so $a>0$ is accordingly fixed).
\end{remark}

\begin{proof}[Proof of Theorem \ref{teo4}]
We prove Theorem \ref{teo4} by means of Proposition \ref{prop72}. Observe that condition \eqref{eq737} in Remark \ref{rem72} is fulfilled.
Set
$$
\zeta(t)=(T+t)^{-\alpha}\,, \quad \eta(t)=(T+t)^{-\beta}, \quad \text{for all } \quad t>0\,.
$$
Consider conditions \eqref{eq735}, \eqref{eq736} of Proposition \ref{prop72} with this choice of $\zeta(t)$ and $\eta(t)$. Therefore we obtain
\begin{equation}\label{eq740}
\beta (T+t)^{\alpha-1} \ge 0
\end{equation}
and
\begin{equation}\label{eq741}
\begin{aligned}
-\alpha (T+t)^{-\alpha-1}&+\frac{C^{m-1}}{a}\frac{m}{m-1}\frac{\delta(\delta+1)}{c_2}R^{-\delta} (T+t)^{-\alpha m-\beta}\\ &-C^{p-1}(T+t)^{-\alpha p} \ge 0\,.
\end{aligned}
\end{equation}
We take
\begin{equation}\label{alphabeta}
\alpha=\frac{1}{p-1}\,, \quad \beta=\frac{p-m}{p-1}\,.
\end{equation}
Due to \eqref{alphabeta}, \eqref{eq740} and \eqref{eq741} become
\begin{equation}\label{eq742}
\frac{p-m}{p-1} \ge 0
\end{equation}
\begin{equation}\label{eq743}
\frac{C^{m-1}}{a}\frac{m}{m-1}\frac{\delta(\delta+1)}{c_2}R^{-\delta} \ge C^{p-1} + \frac{1}{p-1}
\end{equation}
Therefore, \eqref{eq735} is verified due to the hypotheses $p>m$ and \eqref{eq736} follows from assumption \eqref{eq739}. Thus the conclusion is obtained by Propositions \ref{prop72} and \ref{cpsup}.
\end{proof}

\subsection{Proof of blow-up for the critical decaying rate}

Suppose \eqref{h1} and \eqref{h2} with $q=2$. To construct a suitable family of subsolution to
\begin{equation}\label{eq70}
u_t -\frac{1}{\rho(x)}(u^m)_{xx}-u^p=0 \quad \text{ in } (-R,R)\times[0,T)=:D_T,
\end{equation}
 we define
\begin{equation}\label{eq71}
\underline u(x,t):=
\begin{cases}
u(x,t) \quad \text{in } [(-R,-R+\varepsilon)\cup(R-\varepsilon, R) ] \times (0,T), \\
v(x,t) \quad \text{in } [-R+\varepsilon, R-\varepsilon] \times (0,T),
\end{cases}
\end{equation}
where
\begin{align}
&u(x,t):=C\zeta(t)\left [1-\frac{\log(R-|x|)}{a}\eta(t)\right]_{+}^{\frac{1}{m-1}},\label{eq72}\\
&v(x,t) := C\zeta(t) \left [ 1-\frac{(R-\varepsilon)^2-|x|^2+\log(\varepsilon)[2\varepsilon(R-\varepsilon)]}{2\varepsilon(R-\varepsilon)} \frac{\eta}{a} \right ]^{\frac{1}{m-1}}_{+}.\label{eq73}
\end{align}
Moreover we set
$$
\begin{aligned}
&F(x,t):= 1-\frac{\log(R-|x|)}{a}\eta(t)\,,\\
&G(x,t):=  1-\frac{|x|^2-(R-\varepsilon)^2+\log(\varepsilon)[2\varepsilon(R-\varepsilon)]}{2\varepsilon(R-\varepsilon)} \frac{\eta(t)}{a} \,.
\end{aligned}
$$

For any $(x,t) \in [(-R,-R+\varepsilon)\cup(R-\varepsilon, R) ] \times (0,T)$, we have:
\begin{align}
&{u}_t =C\zeta ' F^{\frac{1}{m-1}} + \frac{C}{m-1} \zeta F^{\frac{1}{m-1}} - \frac{C}{m-1} \zeta \frac{\eta'}{\eta} F^{\frac{1}{m-1}-1};\label{eq74} \\
&({u}^m)_{xx}=\frac{C^m}{a}\zeta^m\frac{m}{m-1}\frac{\eta}{(R-|x|)^2}F^{\frac{1}{m-1}}+\frac{C^m}{a^2}\zeta^m\frac{m}{(m-1)^2}\frac{\eta^2}{(R-|x|)^2}F^{\frac{1}{m-1}-1}. \label{eq75}
\end{align}
For any $(x,t) \in  [-R+\varepsilon, R-\varepsilon] \times (0,T)$, we have:
\begin{align}
&v_t =C\zeta ' G^{\frac{1}{m-1}} + C\zeta \frac{1}{m-1} \frac{\eta'}{\eta}G^{\frac{1}{m-1}} - C\zeta \frac{1}{m-1} \frac{\eta'}{\eta} G^{\frac{1}{m-1}-1}\,;\label{eq76}\\
&(v^m)_{xx}=\frac{C^m}{a} \zeta^m \frac{m}{m-1} G^{\frac{1}{m-1}} \frac {\eta}{\varepsilon(R-\varepsilon)} + \frac{C^m}{a^2} \zeta^m \frac{m}{(m-1)^2} G^{\frac{1}{m-1}-1} \eta^2 \frac{|x|^2}{\varepsilon^2(R-\varepsilon)^2}\,.\label{eq77}
\end{align}
We also define
\begin{equation}\label{eq78}
\begin{aligned}
&\underline\sigma(t) := \zeta ' + \frac{\zeta}{m-1} \frac{\eta'}{\eta} - \frac{C^{m-1}}{a} \zeta^m \frac{m}{m-1} \frac{\eta}{c_2} ,\\
& \underline\lambda(t) := \frac{\zeta}{m-1} \frac{\eta'}{\eta}\,,  \\
& \underline\gamma(t):=C^{p-1} \zeta^p, \\
& \underline\sigma_0(t) := \zeta'+\frac{\zeta}{m-1} \frac{\eta'}{\eta} - \frac{C^{m-1}}{a} \zeta^{m} \frac{m}{m-1} \frac{\eta}{\varepsilon(R-\varepsilon)\rho_2}, \\
& \textit{K}:= \left (\frac{m-1}{p+m-2}\right)^{\frac{m-1}{p-1}} - \left (\frac{m-1}{p+m-2}\right)^{\frac{p+m-2}{p-1}}>0;
\end{aligned}
\end{equation}
and
\begin{equation}\label{eq70b}
\begin{aligned}
&\text{S}_1:=\left \{ (x,t) \in (-R,-R+\varepsilon)\cup(R-\varepsilon,R) \times (0,T)\,\, |\,\, 0<F(x,t)<1 \right \},\\
&\text{S}_2:= \left \{ (x,t) \in [-R+\varepsilon,R-\varepsilon] \times (0,T) \,\,|\,\, 0<G(x,t)<1\right \}\,.
\end{aligned}
\end{equation}

\begin{proposition}\label{prop71}
Let $p>m>1$. Let $T\in (0,\infty)$, $\zeta$, $\eta \in C^1([0,T); [0, +\infty))$.
Let $\underline\sigma$, $\underline\delta$, $\underline\gamma$, $\underline\sigma_0$, $\textit{K}$ be defined in \eqref{eq78}. Assume that, for all $t\in(0,T)$,
\begin{equation}\label{eq79}
\underline\sigma(t)>0, \quad K[\underline\sigma(t)]^{\frac{p+m-2}{p-1}} \le \underline\lambda(t) \underline\gamma(t)^{\frac{m-1}{p-1}}\,,
\end{equation}
\begin{equation}\label{eq710}
(m-1) \underline\sigma(t) \le (p+m-2) \underline\gamma(t)\,.
\end{equation}
\begin{equation}\label{eq711}
\underline\sigma_0(t)>0, \quad K[{\underline\sigma}_0](t)^{\frac{p+m-2}{p-1}} \le{\underline\lambda}(t){\underline\gamma}(t)^{\frac{m-1}{p-1}}, \end{equation}
\begin{equation}\label{eq712}
(m-1){\underline\sigma}_0(t) \le (p+m-2){\underline\gamma}(t)\,.
\end{equation}
Then $w$ defined in \eqref{eq71} is a  subsolution of equation \eqref{eq70}.
\end{proposition}

\begin{proof}[Proof of Proposition \ref{prop71}]
Let ${u}$ be as in \eqref{eq72} and $B_1$ as in \eqref{eq70b}.
In view of \eqref{eq74}, \eqref{eq75}, we obtain, for all $(x,t)\in S_1$,
\begin{equation}\label{eq713}
\begin{aligned}
&{u}_t - \frac{1}{\rho}({u}^m)_{xx}-{ u}^p\\
&=C\zeta ' F^{\frac{1}{m-1}} + C\zeta \frac{1}{m-1} \frac{\eta'}{\eta}F^{\frac{1}{m-1}} - C\zeta \frac{1}{m-1} \frac{\eta'}{\eta} F^{\frac{1}{m-1}-1}-\frac{1}{\rho(R-|x|)^2} \frac{C^m}{a^2} \zeta^m \frac{m}{(m-1)^2} \eta^2 F^{\frac{1}{m-1}-1} \\ 
& -\frac{1}{\rho(R-|x|)^2} \frac{C^m}{a} \zeta^m \frac{m}{m-1} \eta F^{\frac{1}{m-1}} -C^p \zeta^p F^{\frac{p}{m-1}}\\
&\le C\zeta ' F^{\frac{1}{m-1}} + C\zeta \frac{1}{m-1} \frac{\eta'}{\eta}F^{\frac{1}{m-1}} - \frac{1}{c_2} \frac{C^m}{a} \zeta^m \frac{m}{m-1} \eta F^{\frac{1}{m-1}} -C^p \zeta^p F^{\frac{p}{m-1}}.
\end{aligned}
\end{equation}
where, in view of hypotheses \eqref{h2} with $q=2$, we have used that
\begin{equation*}
-\frac{1}{\rho(R-|x|)^2} \le -\frac{1}{c_2}\,, \quad \text{for all }\,\,\ x\in (-R,-R+\varepsilon)\cup(R-\varepsilon,R)\,.
\end{equation*}
Thanks to \eqref{eq78} and \eqref{eq713}
\begin{equation}\label{eq714}
u_t - \frac{1}{\rho}(u^m)_{xx}- u^p \leq CF^{\frac{1}{m-1}-1}\varphi(F),
\end{equation}
where
\begin{equation}\label{eq715}
\varphi(F):=\underline\sigma(t)F - \underline \lambda(t) - \underline\gamma(t)F^{\frac{p+m-2}{m-1}}.
\end{equation}
Due to \eqref{eq714}, our goal is to find suitable $C>0$, $a>0$, $\zeta$, $\eta$ such that
$$
\varphi(F) \le 0\,, \quad \text{for all}\,\,  F \in (0,1)\,.
$$
Arguing as in the proof of Proposition \eqref{prop2}, the latter reduces to the system
\begin{equation}\label{eq716}
\begin{aligned}
&\varphi(F_0)=K \dfrac{\underline\sigma(t)^{\frac{p+m-2}{p-1}}}{\underline\gamma(t)^{\frac{m-1}{p-1}}} - \underline\lambda(t) \le 0\,, \\ 
&0<F_0= \left [\frac{m-1}{p+m-2} \frac{\underline\sigma(t)}{\underline\gamma(t)} \right ]^{\frac{m-1}{p-1}} \le 1\,,
\end{aligned}
\end{equation}
where the coefficient $K=K(m,p)$ has been defined in \eqref{eq78}. By hypotheses \eqref{eq79} and \eqref{eq710}, \eqref{eq716} is verified, hence we have proved that
\begin{equation}\label{eq717}
{u}_t-\frac{1}{\rho(x)}({u}^m)_{xx}-{u}^p \le 0 \quad \text{ in } S_1\,.
\end{equation}
Furthermore, since ${u}^m\in C^1([(-R,-R+\varepsilon)\cup(R-\varepsilon,R)]\times[0,T))$, due to Lemma \ref{lemext} (applied with $I_1\times(0,T)=S_1$, $I_2\times(0,T)=\{[(-R,-R+\varepsilon)\cup(R-\varepsilon,R)]\times(0,T)\}\setminus S_1$, $u_1= u$, $u_2=0$, $u= u$), it follows that $u$ is a subsolution to equation
$$
{u}_t-\frac{1}{\rho}({u}^m)_{xx}-{u}^p = 0 \quad \text{ in } [(-R,-R+\varepsilon)\cup(R-\varepsilon,R)]\times (0,T)\,,
$$
in the sense of Definition \ref{def1}.
\smallskip

Now we consider $v$ as in \eqref{eq73}. Due to \eqref{eq76} and \eqref{eq77}, we have, for all $(x,t)\in B_2$,
\begin{equation}\label{eq719}
\begin{aligned}
v_t&-\frac{1}{\rho}(v^m)_{xx}-v^p\\
&\le CG^{\frac 1{m-1}-1}\left \{G\left [\zeta'+\frac{\zeta}{m-1}\frac{\eta'}{\eta}- \frac 1 {\rho_2} \frac{C^{m-1}}{a} \zeta^m \frac m {m-1} \frac{\eta}{\varepsilon(R-\varepsilon)}\right] \right .  \\
& \left . - \frac{\zeta}{m-1}\frac{\eta'}{\eta} - C^{p-1}\zeta^p G^{\frac{p+m-2}{m-1}} \right \},
\end{aligned}
\end{equation}
where we have used that, in view of \eqref{eq21}
$$
-\frac{1}{\rho(x)}\le-\frac{1}{\rho_2} \quad \text{for all}\,\,x\in[-R+\varepsilon,R-\varepsilon].
$$
Due to \eqref{eq78} and \eqref{eq719},
\begin{equation}\label{eq720}
v_t - \frac{1}{\rho}(v^m)_{xx}-v^p \leq CG^{\frac 1{m-1}-1} \psi(G),
\end{equation}
where
\begin{equation}\label{eq721}
\psi(G):=\underline\sigma_0(t)G - \underline\lambda(t) - \underline\gamma(t)G^{\frac{p+m-2}{m-1}}.
\end{equation}
Now, by the same arguments used to obtain \eqref{eq712}, in view of \eqref{eq711} and \eqref{eq712} we can infer that
$$
\psi(G)\le 0\, \quad 0<G\le 1\,.
$$
Hence, due to \eqref{eq720}, we have proved that
\begin{equation}\label{eq722}
v_t-\frac{1}{\rho}(v^m)_{xx}-v^p \le 0 \quad \text{ for any }\,\, (x,t)\in S_2\,.
\end{equation}
Moreover, by Lemma \ref{lemext}, $v$ is a subsolution of equation
\begin{equation}\label{eq723}
v_t-\frac{1}{\rho}(v^m)_{xx}-v^p = 0 \quad \text{ in }\,\, (-R+\varepsilon, R-\varepsilon)\times(0,T)\,,
\end{equation}
in the sense of Definition \ref{def1}. Now, observe that $w \in C((-R,R) \times [0,T))$, indeed,
$$
{u}  = v = C \zeta(t)\left [ 1- \log(\varepsilon)\frac{\eta(t)}{a} \right ]_+^{\frac{1}{m-1}} \quad \text{in}\,\, \{R-\varepsilon, -R+\varepsilon\} \times (0,T)\,.
$$
Moreover, $\underline u^m \in C^1((-R,R) \times [0,T))$, indeed, see its definition in \eqref{eq71}, for any $t\in[0,T)$
$$
\begin{aligned}
&({u}^m)_x(R-\varepsilon,t) = (v^m)_x(R-\varepsilon,t)  = C^m \zeta(t)^m \frac{m}{m-1} \frac{1}{\varepsilon}\frac{\eta(t)}{a} \left [ 1- \log(\varepsilon)\frac{\eta(t)}{a} \right ]_+^{\frac{1}{m-1}} ,\\
&({u}^m)_x(-R+\varepsilon,t) = (v^m)_x(-R+\varepsilon,t)  = -C^m \zeta(t)^m \frac{m}{m-1} \frac{1}{\varepsilon}\frac{\eta(t)}{a} \left [ 1- \log(\varepsilon)\frac{\eta(t)}{a} \right ]_+^{\frac{1}{m-1}}.
\end{aligned}
$$
Hence, by Lemma \ref{lemext} again, $\underline u$ is a subsolution to equation \eqref{eq70} in the sense of Definition \ref{def1}.
\end{proof}

\begin{remark}\label{rem71}
Let $$p>m\,,$$ and assumptions \eqref{h1} and \eqref{h2} with $q=2$ be satisfied. Let define $\omega:= \frac{C^{m-1}}{a}$. In Theorem \ref{teo4}, the precise hypotheses on parameters $C>0$, $a>0$, $\omega>0$ and $T>0$ is the following.
\begin{equation}\label{eq724}
\max \left \{1-\frac{C^{m-1}}{a} \frac{m}{c_2}\,; 1-\frac{C^{m-1}}{a} \frac{m}{\rho_2}\frac{1}{\varepsilon(R-\varepsilon)}\right \} \le (p+m-2)C^{p-1}\,,
\end{equation}
\begin{equation}\label{eq725}
\begin{aligned}
\textit{K} \max& \left \{ \left [\frac1{m-1}-\frac{C^{m-1}}{a} \frac{m}{m-1}\frac{1}{c_2} \right ]^{\frac{p+m-2}{p-1}};\left [\frac1{m-1}-\frac{C^{m-1}}{a} \frac{m}{m-1}\frac{1}{\varepsilon(R-\varepsilon)}\frac{1}{\rho_2} \right ]^{\frac{p+m-2}{p-1}} \right \}\\& \,\le \,\frac{p-m}{(m-1)(p-1)}C^{m-1} \,.
\end{aligned}
\end{equation}

\end{remark}

\begin{lemma}\label{lemma71}
All the conditions in Remark \ref{rem71} can hold simultaneously.
\end{lemma}

\begin{proof}
We can take $\omega>0$ such that
$$
\omega_0\le \omega \le \omega_1
$$
for suitable $0<\omega_0<\omega_1$ and we can choose $C>0$ sufficiently large to guarantee \eqref{eq724} and \eqref{eq725} (so, $a>0$ is fixed, too).
\end{proof}

\begin{proof}[Proof of Theorem \ref{teo4}]
We now prove Theorem \ref{teo4}, by means of Proposition \ref{prop71}. In view of Lemma \ref{lemma71} we can assume that all conditions of Remark \ref{rem71} are fulfilled. Set
$$
\zeta=(T-t)^{-\alpha}\,, \quad \eta=(T-t)^{\beta}\,, \quad \text{for all} \quad t>0\,,
$$
where 
$$
\alpha:=\frac 1{p-1}\quad\quad \beta:=\frac{p-m}{p-1}\,.
$$ 
Then the coefficients in \eqref{eq78} read
\begin{equation}
\begin{aligned}
& \underline\sigma(t) = \frac{1}{m-1}\left [ 1-\frac{C^{m-1}}{a}\frac m{c_2}\right] (T-t)^{-\frac{p}{p-1}}\,,\\
&\underline \lambda(t) = \frac{p-m}{(m-1)(p-1)} (T-t)^{-\frac{p}{p-1}}\,,  \\
& \underline\gamma(t)=C^{p-1} (T-t)^{-\frac{p}{p-1}}\,, \\
& \underline \sigma_0(t) =  \frac 1 {m-1} \left [1-\frac{C^{m-1}}{a} \frac{m}{\rho_2}\frac{1}{\varepsilon(R-\varepsilon)}\right](T-t)^{-\frac{p}{p-1}} \,.
\end{aligned}
\label{sost}
\end{equation}

Let $p>m$. Condition \eqref{eq724} implies \eqref{eq79}, \eqref{eq710}, while condition \eqref{eq725} implies \eqref{eq711}, \eqref{eq712}. Hence by Propositions \ref{prop71} and \ref{cpsub} the thesis follows.
\end{proof}

\section{Proof of Theorem \ref{teo5}}\label{subcritical}
We assume \eqref{h1}, \eqref{h2} for $0\le q<2$ and let $d$ be as in \eqref{hd}. In order to construct a suitable family of supersolutions of 
\begin{equation}\label{eq60}
u_t-(u^m)_{xx}\frac{1}{\rho}-u^p=0,
\end{equation}
we define, for all $(x,t)\in (-R,R) \times (0,+\infty)$,
\begin{equation}\label{eq61}
w(x,t):=C\zeta(t)(R-|x|)^{\frac{d}{m}};
\end{equation}
where $\zeta \in C^1([0, +\infty); [0, +\infty))$ and $C > 0$.

\smallskip

For any $(x,t) \in \big[(-R,R)\setminus \{0\}\big]\times (0,+\infty)$, we have:

\begin{equation}\label{eq62}
w_t =C\,\zeta'\, (R-|x|)^{\frac{d}{m}}\,.
\end{equation}
\begin{equation}\label{eq63}
(w^m)_{xx}= d\,(d-1)\,C^m \,\zeta^m \,(R-|x|)^{d-2}\,.
\end{equation}

\begin{proposition}\label{prop61}
Let $\zeta=\zeta(t) \in C^1[0,+\infty); [0, +\infty)), \zeta'\geq 0$. Assume \eqref{h1}, \eqref{hd}, \eqref{h2} for $0\le q<2$, and that
\begin{equation}\label{eq64}
d(d-1)\,\delta\,C^m\,\zeta^m-R^{\frac{p\,d}{m}}C^p\,\zeta^p\,\ge\,0.
\end{equation}
Then $w$ defined in \eqref{eq61} is a supersolution of equation \eqref{eq60}.
\end{proposition}

Here $\delta>0$ is defined as
\begin{equation}\label{eq65}
\delta:=\min\left\{\frac{1}{c_2}\,;\,\frac{R^{1-q}}{c_2}\right\}.
\end{equation}

\begin{proof}[Proof of Proposition \ref{prop61}]
In view of \eqref{eq62} and \eqref{eq63} 
we get, for any $(x,t)\in ((-R,R)\setminus \{0\})\times (0,+\infty)$,
\begin{equation}\label{eq66}
\begin{aligned}
&w_t-\frac 1 {\rho}(w^m)_{xx} -w^p\\
&=C\zeta'(R-|x|)^{\frac{d}{m}} + d(d-1)C^m\,\zeta^m\frac {(R-|x|)^{d-2}}{\rho}- C^p \zeta^p (R-|x|)^{\frac{p\,d}{m}}.
\end{aligned}
\end{equation}
Thanks to hypothesis \eqref{h2} and \eqref{hd} we have
\begin{equation}\label{eq67}
\begin{aligned}
&\frac{(R-|x|)^{d-2}}{\rho} \ge c_1(R-|x|)^{d-2+q}\ge \delta \,,\\
&-(R-|x|)^{\frac{p\,d}{m}} \ge - R^{\frac{p\,d}{m}},
\end{aligned}
\end{equation}
where $\delta$ is defined in \eqref{eq65}.
Furthermore, since $\zeta'\geq 0$, from \eqref{eq66} and \eqref{eq67} we get
\begin{equation}\label{eq68}
w_t-\frac 1 {\rho}(w^m)_{xx} -w^p\ge d(d-1)\,\delta\,C^m\,\zeta^m-R^{\frac{p\,d}{m}}C^p\,\zeta^p\,.
\end{equation}
Hence we get the condition
\begin{equation}\label{eq69}
d(d-1)\,\delta\,C^m\,\zeta^m-R^{\frac{p\,d}{m}}C^p\,\zeta^p\,\ge 0\,,
\end{equation}
which is guaranteed by \eqref{hd} and \eqref{eq64}. Hence we have proved that
$$
w_t-\frac 1 {\rho}(w^m)_{xx} -w^p \ge 0 \quad \text{in }\,\, ((-R,R)\setminus \{0\})\times (0, +\infty)\,.
$$
Now observe that, setting $|x|=:r$,
$$
\begin{aligned}
&w \in C((-R,R)\times [0,+\infty))\,,\\
&w^m \in C^1([(-R,R) \setminus \{0\}]\times [0,+\infty))\,, \\
& (w^m)_{r}(0,t)\le 0\,.
\end{aligned}
$$
Hence, thanks to a Kato-type inequality we can infer that $w$ is a supersolution to equation \eqref{eq60}
in the sense of Definition \ref{def1}.
\end{proof}

\begin{remark}\label{rem61}
Let $$0\le q<2$$ and assumption \eqref{h2} be satisfied. In Theorem \ref{teo5} the precise hypotheses on parameters $\alpha$, $C>0$, $T>0$ are the following.
\begin{itemize}
\item[(a)] Let $p<m$. We require that
\begin{equation}\label{eq69b}
\alpha>0, \quad\quad d(d-1)\,\delta\,C^m\,T^{\alpha (m-p)}\,\ge R^{\frac{p\,d}{m}}C^p \,. 
\end{equation}
\item[(b)] Let $p>m$. We require that
\begin{equation}\label{eq610}
\alpha=0, \quad\quad d(d-1)\,\delta\,C^m\,\ge R^{\frac{p\,d}{m}}C^p\,.
\end{equation}
\end{itemize}
\end{remark}

\begin{lemma}\label{lem61}
All the conditions in Remark \ref{rem61} can hold simultaneously.
\end{lemma}

\begin{proof}
Let $0\le q<2$.
\begin{itemize}
\item[(a)] Observe that, due to \eqref{hd}, the left hand side of \eqref{eq69b} is positive. Therefore, we can select $C>0$ sufficiently large to guarantee \eqref{eq69}.
\item[(b)] Observe that, due to \eqref{hd}, the left hand side of \eqref{eq69b} is positive. Therefore, we can choose $C>0$ sufficiently small to guarantee \eqref{eq610}.
\end{itemize}
\end{proof}

\begin{proof}[Proof of Theorem \ref{teo5}]
In view of Lemma \ref{lem61} we can assume that all conditions in Remark \ref{rem61} are fulfilled.
Set
$$
\zeta(t)=(T+t)^{\alpha}, \quad \text{for all} \quad t \ge 0\,.
$$

Let $p<m$. Inequality \eqref{eq64} reads
$$ d(d-1)\,\delta\,C^m(T+t)^{\alpha\,m}\,- R^{\frac{p\,d}{m}} \,C^p(T+t)^{\alpha\,p} \, \ge 0\quad \textrm{for all }\,\, t>0\,.$$
This latter is guaranteed by \eqref{eq69}, for $T>1$. Hence, by Propositions \ref{prop61} and \ref{prop1} the thesis follows in this case.
\medskip

Let $p>m$. Condition \eqref{eq610} is equivalent to \eqref{eq64}. Hence, by Propositions \ref{prop61} and \ref{prop1} the thesis follows in this case too. The proof is complete.
\end{proof}

\par\bigskip\noindent
\textbf{Acknowledgments.}
The author is members of the Gruppo Nazionale per l'Analisi Matematica, la Probabilit\`a e le loro Applicazioni (GNAMPA, Italy) and of the Istituto Nazionale di Alta Matematica (INdAM, Italy).

\bigskip
\noindent Data sharing not applicable to this article as no datasets were generated or analysed during the current study.


\end{document}